\let\chooseClass1   

	\ifx\chooseClass0
\documentclass[14pt]{extarticle}
	\fi

	\ifx\chooseClass1
\documentclass[12pt]{article}
	\fi

    \ifx\chooseClass2
\documentclass[final,1p,times]{elsarticle}
	\else
\makeatletter
\def\@seccntformat#1{\csname the#1\endcsname.\quad}
\renewcommand\section{\@startsection {section}{1}{\z@}%
                                   {-3.5ex \@plus -1ex \@minus -.2ex}%
                                   {2.3ex \@plus.2ex}%
                                   {\normalfont\large\bfseries}}
\renewcommand\subsection{\@startsection{subsection}{2}{\z@}%
                        {3.25ex plus 1ex minus .2ex}{-.5em}%
                        {\normalfont\normalsize\bfseries}}
\renewcommand\subsubsection{\@startsection{subsubsection}{3}{\z@}%
                        {3.25ex plus 1ex minus .2ex}{-.5em}%
                        {\normalfont\normalsize\it}}
\@addtoreset{equation}{section}
\makeatother
	\fi

\usepackage{amsthm}
\newtheoremstyle{boldhead}
{\topsep}
{\topsep}
{\slshape}
{}
{\bfseries}
{.}
{ }
{\thmname{#1}\thmnumber{ #2}\thmnote{ (#3)}}

\newtheoremstyle{boldremark}
{\topsep}
{\topsep}
{\upshape}
{}
{\bfseries}
{.}
{ }
{\thmname{#1}\thmnumber{ #2}\thmnote{ (#3)}}

\swapnumbers

\theoremstyle{boldhead}
\newtheorem{theorem}[subsection]{Theorem}
\newtheorem{claim}[subsection]{Claim}
\newtheorem{lemma}[subsection]{Lemma}
\newtheorem{proposition}[subsection]{Proposition}

\theoremstyle{boldremark}
\newtheorem*{acknowledgement}{Acknowledgement}
\newtheorem{conjecture}[subsection]{Conjecture}
\newtheorem{definition}[subsection]{Definition}
\newtheorem{example}[subsection]{Example}
\newtheorem{remark}[subsection]{Remark}

\usepackage{    amsmath,
                amsfonts,
                amssymb,
}

\numberwithin{equation}{section}

\usepackage{t-angles} 
\IfFileExists{url.sty}{\usepackage{url}}{}
\providecommand{\url}[1]{{\tt #1}}

\usepackage{ifpdf}
\ifpdf
    \IfFileExists{hyperref.sty}{\usepackage[pdftex]{hyperref}}{}
\else
    \IfFileExists{hyperref.sty}{\usepackage[hypertex]{hyperref}}{}
\fi

\IfFileExists{euscript.sty}{\usepackage[mathcal]{euscript}}{}
\IfFileExists{mathrsfs.sty}{\usepackage{mathrsfs}}{\let\mathscr\mathfrak}

\IfFileExists{dsfont.sty}%
{\usepackage[sans]{dsfont}\newcommand\1{{\mathds 1}}}%
{\newcommand\1{{1\mkern-5mu {\mathrm I}}}}

\usepackage{diagrams}
\diagramstyle[height=2em,balance,righteqno,PostScript=dvips,nohug]

\newarrowtail{sscurlyvee}{\raise.18ex\hbox{$\sss\succ$}}%
{\mkern-.5mu\raise.18ex\hbox{$\sss\prec$}\mkern.4mu}%
{\sss\curlyvee}{\sss\curlywedge}

\newarrow{Cof}{sscurlyvee}---{->}
\newarrow{Fib}----{->>}
\newarrow{Epi}----{triangle}
\newarrow{Mono}{boldhook}---{->}
\newarrow{TTo}----{->}

\newcommand\KK{{\mathbb K}}
\newcommand\NN{{\mathbb N}}
\newcommand\ZZ{{\mathbb Z}}

\newcommand{\ca}{{\mathcal A}}
\newcommand{\cb}{{\mathcal B}}
\newcommand{\cc}{{\mathcal C}}
\newcommand{\cd}{{\mathcal D}}
\newcommand{\cf}{{\mathcal F}}
\newcommand{\cg}{{\mathcal G}}
\newcommand{\CG}{{\mathscr G}}
\newcommand{\cl}{{\mathcal L}}
\newcommand{\cn}{{\mathcal N}}
\newcommand{\co}{{\mathcal O}}
\newcommand{\cp}{{\mathcal P}}
\newcommand{\cP}{{\mathcal P}}
\newcommand{\cq}{{\mathcal Q}}
\newcommand{\cR}{{\mathcal R}}
\newcommand{\cs}{{\mathcal S}}
\newcommand{\cu}{{\mathcal U}}
\newcommand{\cv}{{\mathcal V}}
\newcommand{\cw}{{\mathcal W}}
\newcommand{\cx}{{\mathcal X}}
\newcommand{\sff}{{\mathsf f}}

\newcommand{\ainf}[1]{$A_\infty$\nobreakdash-\hspace{0pt}}
\newcommand{\colim}{\qopname\relax m{colim}}

\newcommand{\imodi}{\textup{-mod-}}
\newcommand{\modul}{\textup{-mod}}

\newcommand{\n}[1]{\nobreakdash-\hspace{0pt}}
\newcommand{\nOp}[1]{{}_{#1}\mkern-4.5mu\Op}
\newcommand{\sS}[2]{\vphantom{#2}#1 #2}
\newcommand{\Sim}{{\mkern-4.5mu\sim\mkern1.5mu}}
\newcommand{\su}{{\mathsf{su}}}
\newcommand{\tdt}{\otimes\dots\otimes}

\let\emptyset\varnothing
\let\eps\varepsilon
\let\ge\geqslant
\let\kk\Bbbk
\let\le\leqslant

\let\sss\scriptstyle
\let\tens\otimes
\let\ttt\textstyle
\let\und\underline

\DeclareMathOperator\as{\textit{as}}
\DeclareMathOperator\As{\textit{As}}

\DeclareMathOperator\ass{\textit{as1}}
\DeclareMathOperator\Ass{\textit{As1}}

\DeclareMathOperator\Coker{Coker}
\DeclareMathOperator\Col{Col}

\DeclareMathOperator\Cone{Cone}
\DeclareMathOperator\dg{\mathbf{dg}}
\DeclareMathOperator\udg{\underline{\dg}}

\DeclareMathOperator\END{{\mathcal E}\textit{nd}}

\DeclareMathOperator\gr{\mathbf{gr}}

\DeclareMathOperator\HOM{{\mathcal H}\textit{om}}
\DeclareMathOperator\id{id}
\DeclareMathOperator\Id{Id}
\DeclareMathOperator\im{Im}
\DeclareMathOperator\inj{in}
\DeclareMathOperator\Inp{Inp}
\DeclareMathOperator\IV{v}

\DeclareMathOperator\Ob{Ob}

\DeclareMathOperator\Op{Op}

\DeclareMathOperator\Set{\mathcal Set}

\newcommand{\propref}[1]{Proposition~\ref{#1}}
\newcommand{\remref}[1]{Remark~\ref{#1}}

\newcommand{\thmref}[1]{Theorem~\ref{#1}}

\begin{document}

    \ifx\chooseClass2
\begin{frontmatter}
\title{Homotopy unital $A_\infty$-algebras}
\author{Volodymyr Lyubashenko\fnref{VLthanks}}
\ead{lub@imath.kiev.ua}
\ead[url]{http://www.math.ksu.edu/$\sim$lub}
\address{Institute of Mathematics,
Nat. Acad. Sci. Ukraine,
3 Tereshchenkivska st.,
Kyiv-4, 01601 MSP,
Ukraine}
\fntext[VLthanks]{The research was supported in part by research program 
0107U002333 of National Academy of Sciences of Ukraine}
\begin{abstract}
It is well known that the differential graded operad of $A_\infty$\n-algebras is a cofibrant replacement (a $\dg$\n-resolution) of the operad of associative differential graded algebras without units. 
In this article we find a cofibrant replacement of the operad of associative differential graded algebras with units. 
Algebras over it are called homotopy unital $A_\infty$\n-algebras. 
We prove that the operad bimodule of $A_\infty$\n-morphisms is a cofibrant replacement of the operad bimodule of morphisms of $\dg$\n-algebras without units. 
Similarly we show that the operad bimodule of homotopy unital 
$A_\infty$\n-morphisms is a cofibrant replacement of the operad bimodule of morphisms of $\dg$\n-algebras with units.
\end{abstract}

\begin{keyword}
$A_\infty$-algebras \sep $A_\infty$-morphisms \sep operads

\MSC 18D50
\end{keyword}

\end{frontmatter}
	\else
\title{Homotopy unital $A_\infty$-algebras}
\author{Volodymyr Lyubashenko}
\date{lub@imath.kiev.ua
\thanks{Institute of Mathematics,
National Academy of Sciences of Ukraine,
3 Tereshchenkivska st.,
Kyiv-4, 01601 MSP,
Ukraine}
\thanks{The research was supported in part by research program 0107U002333 of 
National Academy of Sciences of Ukraine}
}
\maketitle
\begin{abstract}
It is well known that the differential graded operad of $A_\infty$\n-algebras is a cofibrant replacement (a $\dg$\n-resolution) of the operad of associative differential graded algebras without units. 
In this article we find a cofibrant replacement of the operad of associative differential graded algebras with units. 
Algebras over it are called homotopy unital $A_\infty$\n-algebras. 
We prove that the operad bimodule of $A_\infty$\n-morphisms is a cofibrant replacement of the operad bimodule of morphisms of $\dg$\n-algebras without units. 
Similarly we show that the operad bimodule of homotopy unital 
$A_\infty$\n-morphisms is a cofibrant replacement of the operad bimodule of morphisms of $\dg$\n-algebras with units.
\end{abstract}
	\fi

\allowdisplaybreaks[1]

\begin{flushright}
\it To Corrado De Concini on his 60-th birthday
\end{flushright}

The notion of an \ainf-algebra is known from the pioneer work of Stasheff \cite{Stasheff:HomAssoc}. 
Markl~\cite{Markl:ModOp} gave an explanation why \ainf-algebras are defined in their peculiar way. 
He showed that the $\dg$\n-operad $A_\infty$ of \ainf-algebras is a resolution of the $\dg$\n-operad $\As$ of associative non-unital $\dg$\n-algebras. 
The $\dg$\n-operad $A_\infty$ is freely generated by $n$\n-ary operations $m_n$ of degree $2-n$ for $n\ge2$. 
We reformulate these observations stating that the operad projection \(A_\infty\to\As\) is a homotopy isomorphism and a cofibrant replacement for certain model structure of the category of $\dg$\n-operads introduced by Hinich~\cite{Hinich:q-alg/9702015}. 
In this model structure weak equivalences are quasi-isomorphisms, fibrations are epimorphisms of collections, and freely generated operads are examples of cofibrant objects.

Units for \ainf-algebras are as important as they are for associative algebras. 
However, they are much harder to define except the notion of a strict unit, which occurs naturally only in $\dg$\n-categories, but is technically useful. 
Let us mention that there are at least three different definitions dealing with non-strict units: homotopy unital \ainf-algebras due to Fukaya~\cite{Fukaya:FloerMirror-II} (see also \cite{FukayaOhOhtaOno:Anomaly}), unital \ainf-algebras due to the author~\cite{Lyu-AinfCat}, and weakly unital \ainf-algebras due to Kontsevich and Soibelman~\cite{math.RA/0606241}. 
All these definitions apply also to \ainf-categories. 
Among those \ainf-algebras unital ones require minimal data besides the 
\ainf-structure: a non-strict unit and two homotopies. 
More involved definition of a homotopy unital \ainf-algebra assumes a large coherent system of homotopies. 
It applies to a wider class of filtered \ainf-algebras. 
Fortunately, due to results of Manzyuk and the author \cite{LyuMan-unitalAinf} all three definitions are equivalent in the following sense: an arbitrary \ainf-algebra is unital iff it admits a homotopy unital structure iff it admits a weakly unital structure.

In this article we find a cofibrant replacement \(A_\infty^{hu}\) of the $\dg$\n-operad $\Ass$ of associative differential graded algebras with units. 
As a graded operad it is freely generated by a nullary homotopy unit $i$ and operations \(m_{n_1;n_2;\dots;n_k}\), $k\ge1$, $n_1$, \dots, $n_k\in\ZZ_{\ge0}$, of arity \(n=\sum_{q=1}^kn_q\), \(n+k\ge3\), and of degree $4-n-2k$. 
The projection morphism \(A_\infty^{hu}\to\Ass\) is a homotopy isomorphism. 
It turns out that \(A_\infty^{hu}\)\n-algebras are precisely homotopy unital \ainf-algebras in the sense of 
Fukaya~\cite{Fukaya:FloerMirror-II,FukayaOhOhtaOno:Anomaly}.

A similar question arises about morphisms of \ainf-algebras: in what sense they form a cofibrant replacement of morphisms of associative algebras? 
We argue that morphisms of algebras over a $\dg$\n-operad come from bimodules over this operad, and bimodules form a model category. 
The \ainf-bimodule $F_1$ describing \ainf-morphisms is freely generated by $n$\n-ary elements $\sff_n$ of degree $1-n$, $n\ge1$, interpreted as linear maps \(\sff_n:A^{\tens n}\to B\) for two \ainf-algebras $A$ and $B$. 
Thus, $F_1$ encodes the usual notion of an \ainf-morphism. 
We prove that the \ainf-bimodule $F_1$ is a cofibrant replacement of the regular $\As$\n-bimodule $\As$ describing morphisms of associative 
$\dg$\n-algebras without units. 
The projection \(F_1\to\As\) is a homotopy isomorphism.

We construct also a cofibrant replacement for the regular $\Ass$\n-bimodule $\Ass$ describing morphisms of associative $\dg$\n-algebras with units. 
This \(A_\infty^{hu}\)\n-bimodule $F_1^{hu}$ is freely generated by elements \(\sff_{n_1;n_2;\dots;n_k}\), $k\ge1$, $n_1$, \dots, $n_k\in\ZZ_{\ge0}$, of arity \(n=\sum_{q=1}^kn_q\), \(n+k\ge2\), and degree $3-n-2k$. 
Instead of $\sff_{0;0}$ a special element denoted $v$ is used. 
The projection map \(F_1^{hu}\to\Ass\) is a homotopy isomorphism. 
It turns out that $(A_\infty^{hu},F_1^{hu})$\n-algebra maps can be identified with homotopy unital \ainf-morphisms, which we define entirely in terms of Fukaya's homotopy unital \ainf-data.

\begin{acknowledgement}
The author is grateful to the referee for stimulating questions which led to improvement of the exposition.
\end{acknowledgement}

\section{Operads}
\subsection{Notations.}
We denote by $\NN$ the set of non-negative integers $\ZZ_{\ge0}$. 
Let $\kk$ denote the ground commutative ring. 
Tensor product \(\tens_\kk\) will be denoted simply $\tens$. 
When a $\kk$\n-linear map $f$ is applied to an element $x$, the result is typically written as $x.f$. 
The tensor product of two maps of graded $\kk$\n-modules $f$, $g$ of certain degree is defined so that for elements $x$, $y$ of arbitrary degree
\[ (x\tens y).(f\tens g) = (-1)^{\deg y\cdot\deg f}x.f\tens y.g.
\]
In other words, we strictly follow the Koszul rule. 
Composition of $\kk$\n-linear maps $X\xrightarrow{f}Y\xrightarrow{g}Z$ is denoted $f\cdot g:X\to Z$.

In general, $\cv$ denotes a cocomplete symmetric closed monoidal category. 
However we shall use only the category of sets $\cv=(\Set,\times)$ at the very beginning, the category of graded $\kk$\n-modules $\cv=(\gr,\tens_\kk)$ and the category of complexes of $\kk$\n-modules (differential graded $\kk$\n-modules) $\cv=(\dg,\tens_\kk)$ for the rest of the article. 
So we adapt our notation to the latter case, changes which should be made for general $\cv$ being obvious.

Informally, non-symmetric operads are collections of operations, which can be performed without permuting the arguments, in algebras of a certain type. 
In this article an operad will mean a \emph{non-symmetric} operad. 
In order to define $\cv$\n-operads we consider the category $\Col\cv=\cv^\NN$ of collections \((\cw(n))_{n\in\NN}\) in $\cv$, equipped with the tensor product $\odot$:
\[ (\cu\odot\cw)(n) = \bigoplus_{n_1+\dots+n_k=n}^{k\ge0}
\cu(n_1)\tens\dots\tens\cu(n_k)\tens\cw(k).
\]
Here $\NN$ is viewed as a discrete category. 
Elements of $\cw(n)$ are said to be $n$\n-ary or to have arity $n$. 
The unit object of $\Col\cv$ is $\1=\kk$ with \(\1(1)=\kk\), $\1(n)=0$ for $n\ne1$.

\begin{definition}
A (non-symmetric) $\dg$\n-operad $\co$ is a monoid in $\Col\dg$,
\((\co,\mu:\co\odot\co\to\co,\eta:\1\to\co)\). 
Morphisms of $\dg$\n-operads are morphisms of monoids in $\Col\dg$. 
The category of $\dg$\n-operads is denoted $\Op$.
\end{definition}

The associative multiplication $\mu$ consists of chain maps
\[ \mu_{n_1,\dots,n_k}:\co(n_1)\tdt\co(n_k)\tens\co(k) \to \co(n_1+\dots+n_k).
\]
The unit $\eta$ identifies with a cycle in \(\co(1)^0\). 
The full name of a $\dg$\n-operad is a differential graded operad.

For any object $X$ of $\cv$ there is the operad of its endomorphisms
\((\END X)(n)=\und\cv(X^{\tens n},X)\). 
For a complex \(X\in\Ob\dg\) the complex \((\END X)(n)=\und\dg(X^{\tens n},X)\) consists of linear maps \(X^{\tens n}\to X\) of certain degree.

\begin{definition}
An algebra $X$ over a $\cv$\n-operad $\co$ is an object $X$ of $\cv$ together with a morphism of operads \(\co\to\END X\).
\end{definition}

\begin{example}
The $\dg$\n-operad $\As$ is the $\kk$\n-linear envelope of the operad $\as$ in $\Set$. 
They have \(\as(0)=\emptyset\), \(\As(0)=0\) and $\as(n)=\{m^{(n)}\}$, $\As(n)=\kk m^{(n)}=\kk$ for $n>0$. 
The identity operation $m^{(1)}$ is the unit of the operad, and $m^{(2)}=m$ is the binary multiplication.
$\as$\n-algebras are semigroups without unit, and $\As$\n-algebras are associative differential graded $\kk$\n-algebras without unit.

Similarly, the operad $\ass$ in $\Set$ with $\ass(n)=\{m^{(n)}\}$ for all $n\ge0$ has the $\kk$\n-linear envelope -- the $\dg$\n-operad $\Ass$ with 
$\Ass(n)=\kk m^{(n)}=\kk$ for all $n\ge0$. 
Clearly, $\ass$\n-algebras are semigroups with unit implemented by the nullary operation $m^{(0)}$, and $\Ass$\n-algebras are associative differential graded $\kk$\n-algebras with multiplication $m^{(2)}$ and unit $m^{(0)}=1^\su$.
\end{example}

\subsection{Free operads.}
The underlying functor from the category $\Op$ of $\dg$\n-operads to the category of collections has a left adjoint \(T:\Col\dg\rightleftarrows\Op:U\). 
Let us discuss the details, which we use for proving that $\Op$ has a model structure. 
There are at least two distinct ways to describe free operads. 
Let us present the first approach. 

The tensor collection of a collection $\cw$ is defined as \(T_\odot\cw=\oplus_{n\ge0}\cw^{\odot n}\). 
The free operad $T\cw$ generated by $\cw$ is given by \(T\cw=T_\odot(\cw_+)/\sim\), where \(\cw_+=\cw\oplus\1\), the $\kk$\n-linear equivalence relation identifies summands
 \(\cw_+^{\odot n}=\cw_+^{\odot n}\odot\1\rMono^{1\odot\bar\eta}
 \cw_+^{\odot n}\odot\cw_+\simeq\cw_+^{\odot(n+1)}\)
and \(\cw_+^{\odot n}\), where \(\bar\eta=\inj_2(1)\), 
\(\inj_2:\1\to\cw\oplus\1\); thus, \(x\tens\bar\eta\sim x\) for any 
\(x\in T_\odot(\cw_+)\). 
Furthermore, if the collection 
\(\cw_+^{\odot a}\odot\bigl((\cw_+\odot\cw_+)\odot\cw_+^{\odot b}\bigr)\) contains inside \(\cw_+\odot\cw_+\) one of the subexpressions
\((\bar\eta\tdt\bar\eta)\tens w_+\), \(w_+\tens\bar\eta\) with \(w_+\in\cw_+\),
then one can interchange these subexpressions. 
The first relation identifies, in particular, 
\(\im\bar\eta\subset T_\odot^1(\cw_+)\) with \(T_\odot^0(\cw_+)\). 
The second relation implies a similar identification inside the factor \(\cw_+^{\odot(c+1)}\) of 
\(\cw_+^{\odot a}\odot(\cw_+^{\odot(c+1)}\odot\cw_+^{\odot b})\) for an arbitrary \(w_+\in\cw_+^{\odot c}\). 
As a consequence, \(\cw_+^{\odot a}\odot(\1\odot\cw_+^{\odot b})\) is identified with \(\cw_+^{\odot a}\odot\cw_+^{\odot b}\).

The multiplication
\[ \cw_+^{\odot m_1}(n_1)\tdt\cw_+^{\odot m_k}(n_k)\tens\cw_+^{\odot l}(k)
\to \cw_+^{\odot(m+l)}(n_1+\dots+n_k)
\]
is obtained by choosing an integer \(m\ge m_1,\dots,m_k\), adding several factors \(\bar\eta\) in order to embed the source into the direct summand
\[ \cw_+^{\odot m}(n_1)\tdt\cw_+^{\odot m}(n_k)\tens\cw_+^{\odot l}(k)
\hookrightarrow \cw_+^{\odot(m+l)}(n_1+\dots+n_k).
\]
In $T\cw$ the result does not depend on $m$, which proves that $T\cw$ is an operad with the unit
 \(\bigl(\1=T_\odot^0(\cw_+)\to T\cw\bigr)=
 \bigl(\1\rTTo^{\bar\eta} T_\odot^1(\cw_+)\to T\cw\bigr)\).

In the second approach the same free operad is presented via trees. 
A \emph{rooted tree} is a connected non-oriented graph with one distinguished vertex (the root) and without cycles (in particular, without multiple edges or loops). 
For a rooted tree \(t=(V(t),E(t))\) there is a partial ordering on the set of all vertices $V(t)$, namely, $u\preccurlyeq v$ iff $v$ lies on the simple path connecting $u$ with the root. 
This ordering makes $t$ into an oriented graph: an edge \((u,v)\) gets the orientation $u\to v$ iff $u\preccurlyeq v$. 
Thus the rooted tree is oriented towards the root. 
An \emph{ordered rooted tree} is a rooted tree with a chosen total ordering of the set of incoming edges for each vertex. 
Up to an isotopy there is only one embedding of an ordered rooted tree into the oriented plane with coordinates $x$, $y$, which agrees with the orientation and such that each predecessor vertex $u$ is higher than its successor vertex $v$ (has bigger coordinate $y$). 
Thus, smaller incoming edges are to the left of the bigger ones.
A \emph{rooted tree with inputs} is a rooted tree $t$ with a chosen subset $\Inp(t)$ of the set $L(t)$ of leaves, vertices without incoming edges. 
For instance, a 1-vertex tree $t^1$ has one leaf -- the root. 
The set of internal vertices is defined as \(\IV(t)=V(t)-\Inp(t)\). 
For example, the picture
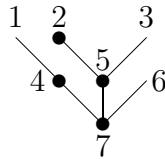
\[
\hstretch140
\vstretch70
\begin{tanglec}
\object1\step\object2\Step\object3\step
\\
\node\Step
\\
\nw1\Put(-5,0)[0,0]{\mbox{\normalsize4}}\node\nw1\node
\Put(0,10)[0,0]{\mbox{\normalsize5}}\ne1
\Put(3,0)[0,0]{\mbox{\normalsize6}}\step
\\
\nw1\n\ne1
\\
\object7
\end{tanglec}
\]
describes an ordered rooted tree with inputs $t$ with the root 7, nullary vertex (leaf) 2, unary vertex~4, the set of leaves $L(t)=\{1,2,3,6\}$, the subset of inputs $\Inp(t)=\{1,3,6\}$, non-input leaf~2 and the set of internal vertices \(\IV(t)=\{2,4,5,7\}\).

Extending the standard terminology, we define a \emph{recursive rooted tree with inputs} as a rooted tree with inputs $t$ equipped with a total ordering $\le$ of the set of internal vertices \(\IV(t)\) such that $u\preccurlyeq v$ implies $u\le v$ for any two internal vertices $u$, $v$. 
We say that $\le$ is admissible.

With any rooted tree with inputs $t$ we associate a groupoid \(\cg(t)\), whose objects are admissible total orderings of the set of internal vertices \(\IV(t)\).
By definition between any two objects of $\cg(t)$ there is precisely one morphism.
Thus \(\cg(t)\) is contractible (equivalent to the terminal category with one object and one morphism). 
With any collection $\cw$ in the symmetric Monoidal category 
$\cv=(\cv,\tens^I,\lambda^f,\rho^L=\id)=\dg$ 
\cite[Definition~2.5]{BesLyuMan-book} we associate a functor
\(F(\cw,t):\cg(t)\to\cv\), which takes an ordering $\le$ on \(\IV(t)\) to the object \(\tens^{v\in(\IV(t),\le)}\cw(|v|)\), the tensor product over internal vertices ordered accordingly to $\le$, where $|v|$ is the number of incoming edges for the vertex $v$. 
The only morphism between $\le_1$ and $\le_2$ is taken to
\[ \lambda^{\id_{\IV(t)}}: \tens^{v\in(\IV(t),\le_1)}\cw(|v|) \to
\tens^{v\in(\IV(t),\le_2)}\cw(|v|),
\]
which is nothing else but the iterated symmetry of the category $\cv$, changing the order of tensor factors from $\le_1$ to $\le_2$. 
For instance, if $\le_1$ and $\le_2$ differ only by the order of two vertices $u$ and $v$, then $\lambda^{\id_{\IV(t)}}=1\tens c\tens1$, where the symmetry $c$ is applied on neighbour places occupied by $u$ and $v$. 
In the following a \textbf{tree} means an \textbf{isomorphism class of ordered rooted trees with inputs} unless indicated differently.

Define a $\kk$\n-module \(C(\cw,t)=\colim(F(\cw,t):\cg(t)\to\cv)\). 
Since the groupoid \(\cg(t)\) is contractible, the value of $F(\cw,t)$ on any object of $\cg(t)$ (an ordering of $\IV(t)$) can serve as a colimit. 
For a tree $t$ the \emph{tetris ordering} of $\IV(t)\subset V(t)$ can be used: embed the tree into the plane as above so that a direct predecessor $u$ of a vertex $v$ would have the ordinate \(y(u)=y(v)+1\); then order all vertices lexicographically, from the left to the right in the top row, then comes the row below it, until the bottom row is reached with the only vertex -- the root.

Another form of the free operad generated by a collection $\cw$ is $T'\cw$,
\[ T'\cw(n) =\bigoplus_{(t,\Inp(t))}^{|\Inp(t)|=n} C(\cw,t),
\]
the direct sum of $C(\cw,t)$ over all trees with $n$ inputs. 
The multiplication $\mu$ is the isomorphism from 
\(C(\cw,t_1)\tdt C(\cw,t_k)\tens C(\cw,t')\) to $C(\cw,t)$, where $t$ is the grafting of \((t_1,\dots,t_k;t')\) ordered lexicographically:
\[ \IV(t) =\IV(t_1)\bigsqcup_<\dots\bigsqcup_<\IV(t_k)\bigsqcup_<\IV(t').
\]
Roots of $t_1$, \dots, $t_k$ (in that order) are identified with input vertices of $t'$ (the $k$\n-element set \(\Inp(t')\) is totally ordered). 
The unit is given by the 1\n-vertex tree $t^1$ with one input vertex -- the root. 
Its set of internal vertices is empty and \(C(\cw,t^1)=\kk\). 
There is an obvious isomorphism \(T\cw\to T'\cw\).

\subsection{Model category structures.}
The following theorem is proved by Hinich in 
\cite[Section~2.2]{Hinich:q-alg/9702015}, except that he relates a category with the category of complexes $\dg$, not with its power $\dg^S$. 
A minor change of introducing a set $S$ into the statement does not modify essentially Hinich's proof.

\begin{theorem}[Hinich]\label{thm-Hinich-model}
Suppose that $S$ is a set, a category $\cc$ is complete and cocomplete and
\(F:\dg^S\rightleftarrows\cc:U\) is an adjunction. 
Assume that $U$ preserves filtering colimits. 
For any $x\in S$, $p\in\ZZ$ consider the object $\KK[-p]_x$ of $\dg^S$, 
\(\KK[-p]_x(x)=\bigl(0\to\kk\rTTo^1 \kk\to0\bigr)\) (concentrated in degrees $p$ and $p+1$), \(\KK[-p]_x(y)=0\) for $y\ne x$. 
Assume that the chain map \(U(\inj_2):UA\to U(F(\KK[-p]_x)\sqcup A)\) is a quasi-isomorphism for all objects $A$ of $\cc$ and all $x\in S$, $p\in\ZZ$. 
Equip $\cc$ with the classes of weak equivalences (resp. fibrations) consisting of morphisms $f$ of $\cc$ such that $Uf$ is a quasi-isomorphism (resp. an epimorphism). 
Then the category $\cc$ is a model category.
\end{theorem}

We shall recall also several constructions used in the proof of this theorem.
They describe cofibrations and trivial cofibrations in $\cc$. 
Assume that \(M\in\Ob\dg^S\), \(A\in\Ob\cc\), \(\alpha:M\to UA\in\dg^S\).
Denote by \(C=\Cone\alpha=(M[1]\oplus UA,d_{\Cone})\in\Ob\dg^S\) the cone taken pointwise, that is, for any $x\in S$ the complex \(C(x)=\Cone\bigl(\alpha(x):M(x)\to(UA)(x)\bigr)\) is the usual cone.
Denote by \(\bar\imath:UA\to C\) the obvious embedding.
Let \(\eps:FU(A)\to A\) be the adjunction counit.
Following Hinich \cite[Section~2.2.2]{Hinich:q-alg/9702015} define an object \(A\langle M,\alpha\rangle\in\Ob\cc\) as the pushout
\begin{diagram}[w=4em,LaTeXeqno]
FU(A) & \rTTo^\eps &A
\\
\dTTo<{F\bar\imath} &&\dTTo>{\bar\jmath}
\\
FC &\rTTo^g &\NWpbk A\langle M,\alpha\rangle
\label{dia-pushout-A<M-alpha>}
\end{diagram}
If $\alpha=0$, then \(A\langle M,0\rangle\simeq F(M[1])\sqcup A\) and $\bar\jmath=\inj_2$ is the canonical embedding.
We say that $M$ consists of free $\kk$\n-modules if for any $x\in S$, $p\in\ZZ$ the $\kk$\n-module $M(x)^p$ is free.

The proof contains the following important statements.
If $M$ consists of free $\kk$\n-modules and $d_M=0$, then 
\(\bar\jmath:A\to A\langle M,\alpha\rangle\) is a cofibration.
It might be called an \emph{elementary standard cofibration}. If
\[ A \to A_1 \to A_2 \to \cdots
\]
is a sequence of elementary standard cofibrations, $B$ is a colimit of this diagram, then the ``infinite composition'' map $A\to B$ is a cofibration called a \emph{standard cofibration} \cite[Section~2.2.3]{Hinich:q-alg/9702015}.

Assume that \(N\in\Ob\dg^S\) consists of free $\kk$\n-modules, $d_N=0$ and \(M=\Cone\bigl(1_{N[-1]}\bigr)=(N\oplus N[-1],d_{\Cone})\).
Then for any morphism \(\alpha:M\to UA\in\dg^S\) the morphism 
\(\bar\jmath:A\to A\langle M,\alpha\rangle\) is a trivial cofibration in $\cc$ and a standard cofibration, composition of two elementary standard cofibrations.
It is called a \emph{standard trivial cofibration}.
Any (trivial) cofibration is a retract of a standard (trivial) cofibration \cite[Remark~2.2.5]{Hinich:q-alg/9702015}.

When \(F:\dg^S\to\cc\) is the functor of constructing a free $\dg$\n-algebra of some kind, the maps $\bar\jmath$ are interpreted as ``adding variables to kill cycles''.

Hinich uses \thmref{thm-Hinich-model} for describing a model structure of the category of symmetric $\dg$\n-operads \cite{Hinich:q-alg/9702015}. 
I was not able to devise a non-symmetric analogue of his argument around Lemma~6.5.1 [\textit{ibid.}].
Thus, I have chosen a different approach.

\begin{proposition}\label{pro-Operads-model-category}
Define weak equivalences (resp. fibrations) in $\Op$ as morphisms $f$ of $\Op$ such that $Uf$ is a quasi-isomorphism (resp. an epimorphism). 
These classes make $\Op$ into a model category.
\end{proposition}

\begin{proof}
Let us verify that the hypotheses of Hinich's theorem are satisfied for $S=\NN$.
The category $\Op$ of $\dg$\n-operads is cocomplete. 
In particular, the coproduct of $\dg$\n-operads $\ca$ and $\cb$ is a certain quotient of the free operad $T(\ca\oplus\cb)$. 
The free operad is a sum of complexes placed in arity $|\Inp(t)|$
\[ F(\ca,\cb;t)(\le) =\tens^{v\in(\IV(t),\le)} c(v)(|v|)
\]
corresponding to trees $t$, whose internal vertices $v\in\IV(t)$ are coloured in two colours: $c(v)\in\{\ca,\cb\}$ is the colour of $v$. 
A coloured tree $t$ is a \emph{contraction} of a coloured tree $t'$ if the ends of each contracted edge of $t'$ are coloured with the same colour and the colouring of $t$ is induced by that of $t'$. 
In particular, colourless vertices form $\Inp t=\Inp t'$. 
A coloured tree $t'$ is a \emph{reduction} of a coloured tree $t$ if it differs by removing several unary vertices. 
With such a contraction or reduction pair \((t',t)\) we associate a chain map $F(\ca,\cb;t')(\le')\to F(\ca,\cb;t)(\le)$ choosing an admissible total ordering $\le'$ of $\IV(t')$ such that glued vertices of $t'$ are neighbours, taking the induced ordering $\le$ of $\IV(t)$ and applying the composition maps of the corresponding operads. 
For the reduction case we insert units of the corresponding operads. 
Contraction compositions and insertion units form a diagram.
As a category the diagram is generated by contractions and insertions of unary vertices, subject to the relations: any two paths from one coloured tree to another are equal. 
Colimit of this diagram is precisely \(\ca\sqcup\cb\). 
The composition is the grafting followed by contraction composition maps.

Assume that $\cn$ is a $\dg$\n-collection and $\co$ is a $\dg$\n-operad.
Constructing \(T\cn\sqcup\co\) is easier than the general coproduct: it is the colimit of a diagram of trees coloured with $\cn$ and $\co$, consisting of contraction compositions if both ends of an edge are coloured with $\co$ and insertion units, which means insertion of a unary vertex coloured with $\co$ (that is $\co(1)$). 
Again shape of the diagram is subject to the relations: any two paths from one coloured tree to another are equal. 
Observe that the diagram is a disjoint union of connected components and each of them has a terminal object. 
The following list $\cl$ contains precisely one terminal object from each connected component: coloured trees without edges whose ends have the same colour. 
We call such trees \emph{terminal}. 
Since an operad is an algebra in an appropriate monoidal category, any two morphisms with the same source and target composed of multiplication in an operad and insertion of units coincide. 
Therefore, the colimit along a connected component is isomorphic to
\[ C(\cn,\co;t) =\colim_{\le\in\cg(t)} F(\cn,\co;t)(\le)
\]
for the terminal object $t$. Hence,
\begin{equation}
T\cn\sqcup\co =\bigoplus_{t\in\cl} C(\cn,\co;t),
\label{eq-TNO=C(NOt)}
\end{equation}
the sum over all terminal trees $t$.

Assume that $\cn$ is a contractible $\dg$\n-collection. 
For instance, $\cn=\KK[-p]_x$ is contractible for $x\in\NN$, $p\in\ZZ$. 
Let us prove that the operad morphism \(\alpha=\inj_2:\co\to T\cn\sqcup\co\) is homotopy invertible. 
In fact, the morphism of operads \(\beta:T\cn\sqcup\co\to\co\), 
\(\beta\big|_\cn=0\), \(\beta\big|_\co=1\) gives \(\alpha\cdot\beta=1_\co\), and \(g=\beta\cdot\alpha:T\cn\sqcup\co\to T\cn\sqcup\co\) is homotopic to \(f=1_{T\cn\sqcup\co}\) in the $\dg$\n-category of collections, as we show next.
There is a degree $-1$ map \(h:\cn\to\cn\) such that \(dh+hd=1_\cn\). 
It is extended by 0 to the map \(h'=h\oplus0:\cn\oplus\co\to\cn\oplus\co\), which satisfies \(dh'+h'd=f-g:\cn\oplus\co\to\cn\oplus\co\). 
For any terminal coloured tree $t$ the morphisms $f$, $g$ preserve the summand
\(C(\cn,\co;t)\simeq F(\cn,\co;t)(\le)\) with some chosen admissible ordering of $\IV(t)$. 
These morphisms are obtained by applying \(f\big|_\cn=1:\cn\to\cn\), 
\(f\big|_\co=1:\co\to\co\), \(g\big|_\cn=0:\cn\to\cn\), 
\(g\big|_\co=1:\co\to\co\) to each \(\cn(k)\) or \(\co(k)\) corresponding to an internal vertex of $t$. 
Define a $\kk$\n-endomorphism of degree $-1$
\begin{multline*}
\hat{h} =\sum_{v\in(\IV(t),\le)} f\tdt f\tens h'\tens g\tdt g:
\\
F(\cn,\co;t)(\le) \to F(\cn,\co;t)(\le) =\tens^{v\in(\IV(t),\le)}c(v)(|v|),
\end{multline*}
where $h'$ is applied on the place indexed by $v$. Then
\[ d\hat{h} +\hat{h}d =\sum_{v\in(\IV(t),\le)} f\tdt f\tens(f-g)\tens g\tdt g
=f\tdt f -g\tdt g =f-g,
\]
therefore, $f$ and $g$ are homotopic to each other.
\end{proof}

General definition of a cofibrant replacement takes in $\Op$ the following form: a cofibrant replacement of a $\dg$\n-operad $\co$ is a trivial fibration (a fibration and a weak equivalence simultaneously) \(\ca\to\co\) such that the unit \(\eta:\1=\kk\to\ca\) is a cofibration in $\Op$.

\begin{example}
Using Stasheff's associahedra one proves that there is a cofibrant replacement \(A_\infty\to\As\) where the graded operad $A_\infty$ is freely generated by $n$\n-ary operations $m_n$ of degree $2-n$ for $n\ge2$. 
The differential is found as
\[  m_n\partial =-\sum_{j+p+q=n}^{1<p<n}
(-)^{jp+q}(1^{\tens j}\tens m_p\tens1^{\tens q})\cdot m_{j+1+q}.
\]
Basis \((m(t))\) of \(A_\infty=T\bigl(\kk\{m_n\mid n\ge2\}\bigr)\) over $\kk$ is indexed by isomorphism classes of ordered rooted trees $t$ without unary vertices.
The tree $t^1$ which has just one vertex (the root and the leaf) corresponds to the unit from $A_\infty(1)$.

Algebras over the $\dg$\n-operad $A_\infty$ are precisely \ainf-algebras in the usual sense: complexes $A$ with the differential $m_1$ and operations \(m_n:A^{\tens n}\to A\), \(\deg m_n=2-n\), for $n\ge2$ such that
\[ \sum_{j+p+q=n}(-)^{jp+q}
(1^{\tens j}\tens m_p\tens1^{\tens q})\cdot m_{j+1+q} =0
\]
for all $n\ge2$, for instance, binary multiplication $m_2$ is a chain map, it is associative up to the boundary of the homotopy $m_3$:
\[ (m_2\tens1)m_2 -(1\tens m_2)m_2 =m_3m_1
+(1\tens1\tens m_1 +1\tens m_1\tens1 +m_1\tens1\tens1)m_3,
\]
and so on. 
Moreover, the chain map \(A_\infty(n)\to\As(n)\) is homotopy invertible, for each \(n\ge1\). 
One way to prove it is implied by a remark of Markl 
\cite[Example~4.8]{Markl:ModOp}. 
Another proof uses the operad of Stasheff associahedra \cite{Stasheff:HomAssoc} and the configuration space of $(n+1)$\n-tuples of points on a circle considered by Seidel in his book \cite{SeidelP-book-Fukaya}. 
Details can be found in \cite[Proposition~1.19]{BesLyuMan-book}.
\end{example}

\subsection{Differential graded objects via graded objects.}
Assume that
\begin{diagram}
\dg^S &\pile{\rTTo^F \\ \lTTo_U} &\cc
\\
\dTTo<{P^S} &= &\dTTo>{\tilde P}
\\
\gr^S &\pile{\rTTo^{\bar F} \\ \lTTo_{\bar U}} &\CG
\end{diagram}
is a morphism of adjunctions in the following sense: $P:\dg\to\gr$ is the underlying functor, the functor $\tilde P$ preserves finite colimits, 
\((F,U,\eps:F\circ U\to\Id_\cc,\eta:\Id_{\dg^S}\to U\circ F)\) and 
\((\bar F,\bar U,\bar\eps,\bar\eta)\) are adjunctions such that
\begin{gather*}
\tilde P\circ F =\bar F\circ P^S, \qquad P^S\circ U =\bar U\circ\tilde P,
\\
\tilde P\circ\eps =\bar\eps\circ\tilde P: \tilde P\circ F\circ U
=\bar F\circ\bar U\circ\tilde P \to \tilde P: \cc \to \CG,
\\
P^S\circ\eta =\bar\eta\circ P^S: 
P^S \to P^S\circ U\circ F =\bar U\circ\bar F\circ P^S: \dg^S \to \gr^S.
\end{gather*}
All these assumptions are satisfied, for instance, if we take for $\cc$ the category $\Op$ of $\dg$\n-operads and for $\CG$ the category of 
$\gr$\n-operads.
Assume that \(M\in\Ob\dg^S\), \(A\in\Ob\cc\), \(\alpha:M\to UA\in\dg^S\).
We would like to compute \(\tilde P(A\langle M,\alpha\rangle)\).
Notice that
\[ P^S(\Cone\alpha) =P^S(M[1])\oplus P^SUA =P^SM[1]\oplus\bar U\tilde PA
\]
does not depend on $\alpha$.
Applying $\tilde P$ to pushout square~\eqref{dia-pushout-A<M-alpha>}, we get the pushout square
\begin{diagram}[w=5em]
\bar F\bar U\tilde PA & \rTTo^{\bar\eps\tilde P} &\tilde PA
\\
\dTTo<{\bar F\inj_2} &&\dTTo>{\tilde P\bar\jmath}
\\
\bar F(P^SM[1]\oplus\bar U\tilde PA) &\rTTo^{\tilde Pg} 
&\NWpbk \tilde P(A\langle M,\alpha\rangle)
\end{diagram}
Denoting \(N=P^SM[1]\), \(B=\tilde PA\), we recognize the above diagram as the colimiting cone
\begin{diagram}[w=4em]
\bar F\bar UB &\rEq &\bar F\bar UB &\rTTo^{\bar\eps} &B
\\
\dTTo<{\bar F\inj_2} &&\dTTo>{\inj_2} &&\dTTo>{\inj_2}
\\
\bar F(N\oplus\bar UB) &\rTTo^\sim &\bar FN\sqcup\bar F\bar UB 
&\rTTo^{1\sqcup\bar\eps} &\NWpbk \bar FN\sqcup B
\end{diagram}
Hence, 
	\(\tilde P(A\langle M,\alpha\rangle)\simeq\bar FN\sqcup B
	=\bar F(P^SM[1])\sqcup\tilde PA\).

If \(M[1]=\kk\{x_1,\dots,x_t\}\in\Ob\dg^S\) is a free graded $\kk$\n-module spanned by elements $x_1$, \dots, $x_t$, we denote \(A\langle M,\alpha\rangle\) as \(A\langle x_1,\dots,x_t\rangle\).
In this case 
	\(\tilde P(A\langle M,\alpha\rangle)
	\simeq\bar F(\kk\{x_1,\dots,x_t\})\sqcup\tilde PA\)
is freely generated by $x_1$, \dots, $x_t$ over $\tilde PA$.
This applies to differential graded and graded operads and, as we shall see below, to bimodules over them.
The differential for the differential graded operad (bimodule) 
\(A\langle M,\alpha\rangle\) is recovered in the underlying graded operad (bimodule) \(\tilde P(A\langle M,\alpha\rangle)\) via the Koszul rule.

\subsection{Unital \texorpdfstring{$A_\infty$}{A8}-algebras.}
Now we are interested in finding a cofibrant replacement for the operad $\Ass$ of unital $\dg$\n-algebras. 
This will be interpreted as including non-strict units and associated homotopies into the structure of a \ainf-algebra. 
Approaching unital \ainf-algebras (and categories) we start with strictly unital ones. 
They are governed by the operad $A_\infty^\su$ generated over $A_\infty$ by a nullary degree 0 cycle $1^\su$ subject to the following relations:
\[ (1\tens1^\su)m_2 =1, \quad (1^\su\tens1)m_2 =1, \quad
(1^{\tens a}\tens1^\su\tens1^{\tens b})m_{a+1+b} =0 \text{ \ if \ } a+b>1.
\]
The rows of the following diagram in \(\Col\dg\)
\begin{diagram}[LaTeXeqno]
0 &\rTTo &A_\infty &\rTTo &A_\infty^\su &\rTTo &\kk1^\su &\rTTo &0
\\
&& \dFib<{htis} &&\dFib<{htis}>{p'} &&\dEq
\\
0 &\rTTo &\As &\rTTo &\Ass &\rTTo &\kk1^\su &\rTTo &0
\label{dia-AAk1-AsAs1k1}
\end{diagram}
are exact sequences, split in the obvious way. 
Therefore, the middle vertical arrow $p'$ is a homotopy isomorphism.

There is a standard trivial cofibration and a homotopy isomorphism
 $A_\infty^\su\rCof~\Sim A_\infty^\su\langle1^\su-i,j\rangle=
 A_\infty^\su\langle i,j\rangle$,
where $i$, $j$ are two nullary operations, $\deg i=0$, $\deg j=-1$, with \(i\partial=0\), \(j\partial=1^\su-i\). 
The projection $p'$ decomposes as
\[ p' =\bigl( A_\infty^\su \rCof~\Sim^{htis}
A_\infty^\su\langle i,j\rangle \rTTo^{p''} \Ass\bigr),
\]
where \(p''(1^\su)=1^\su\), \(p''(i)=1^\su\), \(p''(m_2)=m_2\) and other generators go to 0. 
Hence, the projection $p''$ is a homotopy isomorphism as well.

A cofibrant replacement \(A_\infty^{hu}\to\Ass\) is constructed as a 
$\dg$\n-suboperad of \(A_\infty^\su\langle i,j\rangle\) generated as a graded operad by $i$ and $n$\n-ary operations of degree $4-n-2k$
\[ m_{n_1;n_2;\dots;n_k} =(1^{\tens n_1}\tens j\tens1^{\tens n_2}\tens
j\tens\dots\tens1^{\tens n_{k-1}}\tens j\tens1^{\tens n_k})m_{n+k-1},
\]
where \(n=\sum_{q=1}^kn_q\), $k\ge1$, $n_q\ge0$, \(n+k\ge3\).
Let us show that the graded operad \(A_\infty^{hu}\) is free.

First of all, the free graded operad generated by \(m_{n_1;\dots;n_k}\) is embedded into the graded operad \(A_\infty\langle j\rangle\), moreover, there is a split exact sequence of graded collections
\[ 0 \to 
T\bigl(\kk\{m_{n_1;\dots;n_k}
\mid k>0,\, n_q\ge0,\, {k+\ttt\sum_{q=1}^k} n_q\ge3\}\bigr)
\to A_\infty\langle j\rangle \to \kk j \to 0.
\]
In fact, the $\kk$\n-basis of \(A_\infty\langle j\rangle\) is indexed by trees $(t,\Inp(t))$ without unary vertices. 
This includes the tree \(t=\bullet^j\), \(V(t)=\{*\}\), \(\Inp(t)=\emptyset\), which corresponds to the element $j$ itself.
The element of \(A_\infty\langle j\rangle\) corresponding to $(t,\Inp(t))$ is
\[ (1^{\tens n_1}\tens j\tens1^{\tens n_2}\tens j\tdt1^{\tens n_{k-1}}\tens
j\tens1^{\tens n_k})m(t),
\]
where $j$ are put on places from \(L(t)-\Inp(t)\).
On the other hand, the $\kk$\n-basis of 
 \(T\bigl(\kk\{m_{n_1;\dots;n_k}\mid 
 k>0,\,n_q\ge0,\,k+\sum_{q=1}^kn_q\ge3\}\bigr)
 =\kk\bigl\langle m_{n_1;\dots;n_k}
 \mid {k+\ttt\sum_{q=1}^k} n_q\ge3 \bigr\rangle
 \)
is indexed by the same set of trees except for \(m_{0;0}=\bullet^j\) which is excluded by the inequality.
Numbers $n_q$ in the indexing sequence $n_1$; \dots; $n_k$ are represented by $n_q$ edges, and semicolons are replaced with edges starting with non-input leaves.
All these edges have a common end, a vertex representing \(m_{n_1;\dots;n_k}\).

Adding freely $i$ and using \eqref{eq-TNO=C(NOt)} we obtain the split exact sequence of graded collections
\begin{equation}
0 \to 
\kk\bigl\langle m_{n_1;\dots;n_k}, i
\mid k>0,\, n_q\ge0,\, {k+\ttt\sum_{q=1}^k} n_q\ge3 \bigr\rangle
\to A_\infty\langle j,i\rangle \to \kk j \to 0.
\label{eq-0-T(km)-A8ji-kj-0}
\end{equation}
The quotient is still of rank 1, because $j$ is nullary.

Similarly, from the top row of diagram~\eqref{dia-AAk1-AsAs1k1} we deduce the splittable exact sequence of graded collections
\[ 
0\to A_\infty\langle i\rangle\to A_\infty^\su\langle i\rangle\to \kk1^\su\to 0.
\]
We may choose the splitting of this exact sequence as indicated below:
\[ 
0 \to A_\infty\langle i\rangle \to A_\infty^\su\langle i\rangle \to 
\kk\{1^\su-i\} \to 0.
\]
Adding freely $j$ we get from \eqref{eq-TNO=C(NOt)} the split exact sequence
\begin{equation}
0 \to A_\infty\langle i,j\rangle \to A_\infty^\su\langle i,j\rangle \to 
\kk\{1^\su-i\} \to 0.
\label{eq-0-A8ji-A8suji-k1j-0}
\end{equation}
Combining \eqref{eq-0-T(km)-A8ji-kj-0} with \eqref{eq-0-A8ji-A8suji-k1j-0} we get the split exact sequence of graded collections
\[ 0 \to 
\kk\bigl\langle m_{n_1;\dots;n_k}, i
\mid k>0,\, n_q\ge0,\, {k+\ttt\sum_{q=1}^k} n_q\ge3 \bigr\rangle
\to A_\infty^\su\langle i,j\rangle \to \kk\{1^\su-i,j\} \to 0.
\]
The image of the embedding is precisely \(A_\infty^{hu}\), which allows to rewrite the sequence as
\begin{equation}
0 \rTTo A_\infty^{hu} \rTTo^{i'} A_\infty^\su\langle i,j\rangle \rTTo
\kk\{1^\su-i,j\} \rTTo 0.
\label{eq-0-Ahu-Asuij-k1ij-0}
\end{equation}

The differential in $A_\infty^{hu}$ is computed through that of \(A_\infty^\su\langle i,j\rangle\). 
Actually, \eqref{eq-0-Ahu-Asuij-k1ij-0} is a split exact sequence in $\Col\dg$, where the third term obtains the differential \(j\partial=1^\su-i\). 
The third term is contractible, which shows that the inclusion $i'$ is a homotopy isomorphism in $\Col\dg$. 
Hence, the epimorphism \(p=i'\cdot p'':A_\infty^{hu}\to\Ass\) is a homotopy isomorphism as well, where the $\dg$\n-operad $A_\infty^{hu}$ is freely generated by the homotopy unit $i$ (a degree 0 cycle) and operations \(m_{n_1;n_2;\dots;n_k}\) of degree $4-n-2k$.

In order to prove that $\1\to A_\infty^{hu}$ is a standard cofibration we present it as a colimit of sequence of elementary standard cofibrations \(\1=\kk\to\cd_0\to\cd_1\to\cd_2\to\dots\), where
\[ \cd_r =T\bigl(\kk\{i,m_{n_1;\dots;n_k} \mid
\deg m_{n_1;\dots;n_k} \ge-r\}\bigr)
=\kk\langle i,m_{n_1;\dots;n_k} \mid \deg m_{n_1;\dots;n_k} \ge-r \rangle.
\]

An $A_\infty^{hu}$\n-algebra (a homotopy unital \ainf-algebra) is a complex \((A,m_1)\) of $\kk$\n-modules equipped with the homotopy unit (a cycle)
\(i\in A^0\) and operations \(m_{n_1;\dots;n_k}:A^{\tens n}\to A\) subject to relations
\[ [m_{n_1;\dots;n_k},m_1] \equiv m_{n_1;\dots;n_k}\cdot m_1
-(-1)^n\sum_{a+1+c=n}
(1^{\tens a}\tens m_1\tens1^{\tens c})\cdot m_{n_1;\dots;n_k}
=(m_{n_1;\dots;n_k})\partial.
\]
Another (implicit) presentation of this answer due to Fukaya appeared before the question of constructing a cofibrant replacement for $\Ass$.

\begin{definition}[Fukaya, {\cite[Definition~5.11]{Fukaya:FloerMirror-II}}]
A \emph{homotopy unital structure} of an \ainf-algebra $A$ is an 
\ainf-structure $m^+$ of the graded $\kk$\n-module
\[ A^+ = A\oplus\kk1^\su\oplus\kk j
\]
(with \(\deg1^\su=0\), \(\deg j=-1\)) such that:
\begin{enumerate}
\renewcommand{\labelenumi}{(\arabic{enumi})}
\item $A^+$ is a strictly unital \ainf-algebra with the strict unit $1^\su$;

\item the element \(i=1^\su-jm^+_1\) (homotopy unit) is contained in $A$;

\item the embedding \(A\hookrightarrow A^+\) agrees with all \ainf-operations $m_n$, $n\ge1$, in the strict sense;

\item \((A\oplus\kk j)^{\tens n}m^+_n\subset A\), for each $n>1$.
\end{enumerate}
\end{definition}

Notice that degree $-1$ operations \(m_{1,0},m_{0,1}:A\to A\) are unit homotopies, that is,
\[ m_{10}m_1+m_1m_{10} =1-(1\tens i)m_2 \equiv m_{10}\partial, \qquad
m_{01}m_1+m_1m_{01} =1-(i\tens1)m_2 \equiv m_{01}\partial.
\]

While homotopy unitality of an \ainf-algebra \((A,(m_n)_{n\ge1})\) means an \emph{extra structure}, there is another notion of unitality which is a \emph{property} of an \ainf-algebra:

\begin{definition}[{\cite[Definition~7.3]{Lyu-AinfCat}}]
An \ainf-algebra \((A,(m_n)_{n\ge1})\) is unital iff it has a cycle
\(i\in A^0\) such that the following chain maps are homotopic:
\[ (1\tens i)m_2 \sim 1 \sim (i\tens1)m_2: A \to A.
\]
\end{definition}

This definition has proved quite useful. 
It allowed to obtain many results expected for unitality, although to work with it one has to construct some homotopies and to prove acyclicity of certain inductively constructed cones. 
To verify unitality of a given \ainf-algebra is obviously easier than to extend explicitly the given \ainf-structure to a homotopy unital \ainf-structure.
However, the two notions are equivalent in the following sense:

\begin{theorem}[Lyubashenko and Manzyuk 
{\cite[Theorem~3.7]{LyuMan-unitalAinf}}]
An arbitrary \ainf-algebra is unital iff it admits a homotopy unital structure. 
The given \ainf-structure \((A,i,(m_n)_{n\ge1})\) extends to a homotopy unital structure \((A,i,m_{n_1;\dots;n_k})\).
\end{theorem}

\section{Operad bimodules}
Bimodules over operads are of much interest because morphisms of algebras over a $\dg$\n-operad come from such a bimodule. 
An \emph{operad bimodule} is defined as a triple \((\ca;\cp;\cb)\), consisting of operads $\ca$, $\cb$ (monoids in $\Col\dg$) and an $\ca$-$\cb$-bimodule $\cp$ in the sense of monoidal structure of $\Col\dg$. 
Thus, the left and the right actions \(\lambda:\ca\odot\cp\to\cp\) and 
\(\rho:\cp\odot\cb\to\cp\) are associative and commute with each other in the sense of monoidal category $\Col\dg$. 
Notice that the actions
\begin{align*}
\lambda_{n_1,\dots,n_k}:\ca(n_1)\tdt\ca(n_k)\tens\cp(k) &\to\cp(n_1+\dots+n_k),
\\
\rho_{n_1,\dots,n_k}:\cp(n_1)\tdt\cp(n_k)\tens\cb(k) &\to \cp(n_1+\dots+n_k)
\end{align*}
include, in particular, chain maps
\begin{align*}
\lambda_\emptyset =\id: \cp(0) =\kk\tens\cp(0) &\to \cp(0), \text{ for } k=0,
\\
\lambda_0: \ca(0)\tens\cp(1) &\to \cp(0), \text{ for } k=1,\; n_1=0,
\\
\rho_\emptyset: \cb(0) =\kk\tens\cb(0) &\to \cp(0), \text{ for } k=0.
\end{align*}
The category of operad bimodules $\nOp1_1$ has morphisms
\((f;h;g):(\ca;\cp;\cb)\to(\cc;\cq;\cd)\), where \(f:\ca\to\cc\), 
\(g:\cb\to\cd\) are morphisms of $\dg$\n-operads and \(h:\cp\to\sS{_f}\cq_g\) is an $\ca$-$\cb$-bimodule morphism. 
The forgetful functor $U$ from $\ca$-$\cb$-bimodules to $\dg$\n-collections has a left adjoint \(F':\Col\dg\rightleftarrows\ca\imodi\cb:U\). 
Namely, the free $\ca$-$\cb$-bimodule generated by a $\dg$\n-collection $\cx$ is \(F'\cx=\ca\odot\cx\odot\cb\). 
The left and the right actions are \(\lambda=\mu_\ca\odot1_\cx\odot1_\cb\) and \(\rho=1_\ca\odot1_\cx\odot\mu_\cb\).

\begin{definition}
A \emph{distinguished floor} of a tree $t$ is a subset 
\(f(t)\subset\IV(t)=V(t)-\Inp(t)\) such that:
\begin{enumerate}
\renewcommand{\labelenumi}{(\arabic{enumi})}
\item if \(u,v\in f(t)\) and $u\preccurlyeq v$ ($v$ lies on the simple path connecting $u$ with the root), then $u=v$;

\item for any input \(u\in\Inp(t)\) there is \(v\in f(t)\) such that $u\preccurlyeq v$.
\end{enumerate}
\end{definition}

Objects of the category \(\dg^{3\NN}\), \(3\NN=\NN\sqcup\NN\sqcup\NN\), are written as triples of collections of complexes 
\((\cu;\cx;\cw)=(\cu(n);\cx(n);\cw(n))_{n\in\NN}\). 
There is a pair of adjoint functors \(T:\dg^{3\NN}\rightleftarrows\nOp1_1:U\), \(T(\cu;\cx;\cw)=(T\cu;T\cu\odot\cx\odot T\cw;T\cw)\). 
The bimodule part \(B(\cu;\cx;\cw)=T\cu\odot\cx\odot T\cw\) consists of summands indexed by trees $t$ with a distinguished floor $f(t)$ describing \(\cx(n_1)\tdt\cx(n_k)\), \(k=|f(t)|\ge0\), vertices below it labelled by $\cw(\text-)$ and vertices above it labelled by $\cu(\text-)$. 
The bimodule decomposes as
\[ B(\cu;\cx;\cw)(n) =\bigoplus_{(t,\Inp(t))}^{|\Inp(t)|=n} B(\cu;\cx;\cw;t),
\]
summation runs over isomorphism classes of all ordered rooted trees with $n$ inputs and a distinguished floor. Here
\[ B(\cu;\cx;\cw;t) \simeq \tens^{v\in(\IV(t),\le)} c(v)(|v|)
\]
with $c(v)=\cu$, $\cx$ or $\cw$ depending on the position of $v$ -- above, at or below the distinguished floor of $t$.

\begin{proposition}
Define weak equivalences (resp. fibrations) in $\nOp1_1$ as morphisms $f$ of $\nOp1_1$ such that $Uf$ is a quasi-isomorphism (resp. an epimorphism). 
These classes make $\nOp1_1$ into a model category.
\end{proposition}

\begin{proof}
Let us verify that the hypotheses of Hinich's \thmref{thm-Hinich-model} are satisfied for $S=3\NN=\NN\sqcup\NN\sqcup\NN$. 
The category $\nOp1_1$ of $\dg$\n-operad bimodules is cocomplete. 
In particular, the bimodule part $\cR$ of the coproduct of bimodules
\((\ca;\cp;\cb)\sqcup(\cc;\cq;\cd)=(\ca\sqcup\cc;\cR;\cb\sqcup\cd)\) is a certain quotient of the free bimodule 
\(B(\ca\oplus\cc;\cp\oplus\cq;\cb\oplus\cd)\) -- colimit of a diagram, consisting of contraction compositions or actions and of insertions of units of operads.

Taking $x$ from the first, the second or the third copy of $\NN$ in $S$ and denoting \(\cn=\KK[-p]_x\), we represent the morphism 
\(\inj_2:(\ca;\cp;\cb)\to F(\KK[-p]_x)\sqcup(\ca;\cp;\cb)\) in the form
\begin{align}
(\inj_2,\inj,1) &: (\ca;\cp;\cb) \to (T\cn\sqcup\ca;\cq;\cb), \notag
\\
(1,\inj,1) &: (\ca;\cp;\cb) \to (\ca;\cR;\cb), \label{eq-1in1-APB-ARB}
\\
(1,\inj,\inj_2) &: (\ca;\cp;\cb) \to (\ca;\cs;T\cn\sqcup\cb), \notag
\end{align}
Here $\cq$ (resp. $\cR$ or $\cs$) is the quotient of the freely generated bimodule \((T\cn\sqcup\ca)\odot\cp\odot\cb\) (resp. \(\ca\odot(\cn\oplus\cp)\odot\cb\) or \(\ca\odot\cp\odot(T\cn\sqcup\cb)\)). 
Namely, it is a colimit of the diagram, formed by insertion of units of operads and by contraction actions. 
This diagram consists of connected components and each of them has a terminal object. 
The terminal trees $t$ are characterised by the properties:
\begin{enumerate}
\renewcommand{\labelenumi}{\roman{enumi})}
\item contraction actions involving $\ca$, $\cp$, $\cb$ can not start from $t$;
\item insertion of a unary vertex coloured with $\ca$ or $\cb$ breaks the requirement i).
\end{enumerate}
In the first (resp. the third) case the floors below (resp. above) the distinguished level of terminal tree, summand of $\cq$ (resp. $\cs$), are missing. 
In all three cases the bimodule part decomposes into a sum over terminal trees. 
Proceeding as in the proof of \propref{pro-Operads-model-category}, we construct for contractible $\cn$ homotopies $h$ and $\hat{h}$, which shows that maps $\alpha$ from \eqref{eq-1in1-APB-ARB} and maps $\beta$ in the opposite direction taking $\cn$ to 0 are homotopy inverse to each other.
\end{proof}

\begin{example}
Let $X$, $Y$ be complexes. Define a collection \(\HOM(X,Y)\) as
\(\HOM(X,Y)(n)=\udg(X^{\tens n},Y)\). 
Substitution composition \(\HOM(X,Y)\odot\HOM(Y,Z)\to\HOM(X,Z)\) and obvious units \(\1\to\HOM(X,X)\) make the category of complexes enriched in the category $\Col\dg$. 
In particular, \(\END X=\HOM(X,X)\) are algebras in $\Col\dg$ 
($\dg$\n-operads). 
The collection \(\HOM(X,Y)\) is an \(\END X\)-\(\END Y\)-bimodule. 
The multiplication and the actions are induced by the substitution composition. 
Notice that $\rho_\emptyset:\END Y(0)=\udg(\kk,Y)=\HOM(X,Y)(0)$ is the identity map and \(\END Y(0)\simeq Y\).
\end{example}

\begin{definition}
An \emph{algebra map} over a $\dg$\n-operad bimodule \((\ca;\cp;\cb)\) is a pair of complexes $A$, $B$ together with a morphism of operad bimodules
\[ (\ca;\cp;\cb) \to (\END A;\HOM(A,B);\END B).
\]
\end{definition}

In particular, the complexes $A$ and $B$ acquire an $\ca$\n-algebra (resp. $\cb$\n-algebra) structure from the maps \(\ca\to\END A\), \(\cb\to\END B\). 

\textbf{Notation.} We use the shorthand \((\co,\cp)\) for an operad
$\co$\n-bimodule \((\co;\cp;\co)\).

\begin{example}
Consider the operad bimodules \((\As,\As)\), \((\Ass,\Ass)\), where the first term is an operad and the second term is a regular bimodule. 
A morphism of operad bimodules
\[ (\As;\As;\As) \to (\END A;\HOM(A,B);\END B)
\]
amounts to a morphism \(f:A\to B\) of associative differential graded
$\kk$\n-algebras without units. 
The image of the element \(1^F=1\in\As(1)=\kk\) of the bimodule is denoted \(f\in\HOM(A,B)(1)=\udg(A,B)\). 
Being a cycle it is a chain map. 
Action on the regular bimodule implies that $f$ preserves multiplication.

Similarly, a morphism of operad bimodules
\[ (\Ass;\Ass;\Ass) \to (\END A;\HOM(A,B);\END B)
\]
amounts to a morphism \(f:A\to B\) of associative differential graded
$\kk$\n-algebras with units. 
In fact, elements \(1^\su=1\in\Ass(0)=\kk\) of the operad and \(1^\su\rho_\emptyset=1\in\Ass(0)=\kk\) of the bimodule are related by the equation \(1^\su\cdot1^F=\lambda_0(1^\su\tens1^F)=1^\su\rho_\emptyset\).
Coherence with $\rho_\emptyset$
\begin{diagram}[LaTeXeqno]
1^\su\in\Ass(0) &\rTTo^{\rho_\emptyset}_\id &\Ass(0)\ni1^\su\rho_\emptyset
\\
\dTTo &&\dTTo
\\
1_B\in\END B(0) &\rTTo^{\rho_\emptyset}_\id &\HOM(A,B)(0)\ni1_B
\label{dia-As1-As1-EndB-Hom}
\end{diagram}
implies that $1^\su\rho_\emptyset$ is represented by the unit $1_B$. 
Hence the mentioned equation gives \(1_A\cdot f=1_B\) and $f$ is unital.
\end{example}

\subsection{\texorpdfstring{$A_\infty$}{A8}-morphisms.}
Cofibrant replacement of a $\dg$\n-operad bimodule $(\co,\cp)$ is a trivial fibration \((\ca,\cf)\to(\co,\cp)\) such that the only map from the initial bimodule \((\1,0)\to(\ca,\cf)\) is a cofibration in $\nOp1_1$.

\begin{proposition}
There is a cofibrant replacement \((A_\infty,F_1)\) of the operad bimodule $(\As,\As)$, where the bimodule $F_1$ over the operad $A_\infty$ is freely generated by $n$\n-ary elements $\sff_n$ of degree $1-n$, $n\ge1$, thus,  \(F_1=A_\infty\odot\kk\{\sff_n\mid n\ge1\}\odot A_\infty\). 
The differential is given by
\begin{gather}
\sff_k\partial =\sum_{r+n+t=k}^{n>1} 
(-1)^{t+rn}(1^{\tens r}\tens m_n\tens1^{\tens t}) \sff_{r+1+t}
-\sum^{l>1}_{i_1+\dots+i_l=k}
(-1)^\sigma (\sff_{i_1}\tens \sff_{i_2}\tdt \sff_{i_l}) m_l, \notag
\\
\sigma =(i_2-1) +2(i_3-1) +\dots +(l-2)(i_{l-1}-1) +(l-1)(i_l-1).
\label{eq-sigma-sign}
\end{gather}
Moreover, \((A_\infty,F_1)\to(\As,\As)\) is a homotopy isomorphism in \(\dg^{\NN\sqcup\NN}\).
\end{proposition}

\begin{proof}
It is well known that so defined differential $\partial$ for $F_1$ satisfies $\partial^2=0$. 
Generate a free $\As$\n-bimodule $\overline{F}_1$ by $n$\n-ary elements $\sff_n$ of degree $1-n$, thus, 
\(\overline{F}_1=\As\odot\kk\{\sff_n\mid n\ge1\}\odot\As\). 
Actually, \((\As,\overline{F}_1)\) is the coequalizer in $\nOp1_1$ of the pair of morphisms of collections
\[ 0,\inj: (\kk\{(m_2\tens1)m_2-(1\tens m_2)m_2,m_n\mid n\ge3\},0)
\rightrightarrows (A_\infty,F_1),
\]
the second arrow is just the embedding. 
Therefore, the differential in $\overline{F}_1$ reduces to
\begin{gather*}
\sff_k\partial 
=\sum_{r+2+t=k} (-1)^t(1^{\tens r}\tens m\tens1^{\tens t}) \sff_{k-1}
+\sum_{i+j=k} (-1)^j(\sff_i\tens \sff_j)m,
\end{gather*}
and the equation $\partial^2=0$ follows. 
The operad bimodule morphism in question decomposes as
\[ (A_\infty,F_1) \rFib^{htis} (\As,\overline{F}_1) \rFib (\As,\As).
\]
The first epimorphism is a homotopy isomorphism, since $A_\infty\to\As$ is. 
Let us prove the same property for the second epimorphism, that is, the $\As$\n-bimodule epimorphism $p:\overline{F}_1\to\As$ is a homotopy isomorphism in $\dg^\NN$. 
The following is straightforward:

\begin{claim}
For an arbitrary $\dg$\n-operad $\co$ there is a $\dg$\n-category \(\co\modul\), whose objects are left $\co$\n-modules in $\Col\dg$ and degree $t$ morphisms \(f:\cP\to\cq\) are collections of $\kk$\n-linear maps 
\(f(n):\cP(n)\to\cq(n)\) of degree $t$ such that
\begin{diagram}
\co(n_1)\tdt\co(n_k)\tens\cp(k) &\rTTo^\lambda &\cp(n_1+\dots+n_k)
\\
\dTTo<{1\tdt1\tens f} &= &\dTTo>f
\\
\co(n_1)\tdt\co(n_k)\tens\cq(k) &\rTTo^\lambda &\cq(n_1+\dots+n_k)
\end{diagram}
The differential is \(f\mapsto[f,\partial]=f\partial-(-1)^f\partial f\).
\end{claim}

We prove more: the zero degree cycle $p:\overline{F}_1\to\As$ is homotopy invertible in the $\dg$\n-category \(\As\modul\). 
In fact, both left $\As$\n-modules are freely generated: $\overline{F}_1$ by \((\sff_{i_1}\tdt \sff_{i_k})m^{(k)}\), $k\ge1$, and $\As$ by $1^F\in\As(1)$. 
The $\As$\n-module map $p$ is given by
\[ (\sff_{i_1}\tdt \sff_{i_k})m^{(k)}.p =
\begin{cases}
m^{(k)}1^F, &\qquad \text{if } i_1 =\dots =i_k =1,
\\
0, &\qquad \text{otherwise}.
\end{cases}
\]
There is a zero degree cycle \(\beta:\As\to\overline{F}_1\), $1^F\mapsto\sff_1$, in \(\As\modul\). 
Clearly, \(\beta p=1_{\As}\). 
Let us prove that \(p\beta\) is homotopy invertible. 
These two statements would imply that $p$ is homotopy invertible in \(\As\modul\) and $\beta$ is its homotopy inverse.

Define an $\As$\n-module map $h:\overline{F}_1\to\overline{F}_1$ of degree $-1$ by
\[ (\sff_{i_1}\tdt \sff_{i_{k-1}}\tens \sff_{i_k})m^{(k)}.h =
\begin{cases}
(\sff_{i_1}\tdt \sff_{i_{k-2}}\tens \sff_{i_{k-1}+1})m^{(k-1)},
&\quad \text{if } k>1,\, i_k=1,
\\
0, &\quad \text{otherwise}.
\end{cases}
\]
There is a unique zero degree cycle $N$ such that
\[ p\beta =h\partial +\partial h +1-N: \overline{F}_1 \to \overline{F}_1.
\]
Clearly, \(N=1-p\beta+h\partial+\partial h:\overline{F}_1\to\overline{F}_1\). 
It suffices to prove that $N$ is locally nilpotent. 
Namely, if $\overline{F}_1^{(q)}$ is an $\As$\n-submodule generated by \((\sff_{i_1}\tdt \sff_{i_k})m^{(k)}\), $k\le q$, $\overline{F}_1^{(0)}=0$, then \(\overline{F}_1^{(q)}.N\subset\overline{F}_1^{(q-1)}\). 
Therefore, $1-N$ is invertible with the (well--defined) inverse \(\sum_{a=0}^\infty N^a\) and $p\beta$ is homotopy invertible.

The fact that $N$ reduces the filtration by 1 is verified by the explicit formulae:

\begin{lemma}
The $\As$\n-module map $N$ is given by $\sff_n.N=0$ for $n\ge1$, and for $k>1$
\begin{gather*}
(\sff_{i_1}\tdt \sff_{i_k})m^{(k)}.N =0, \qquad \text{if } i_k>2,
\\
(\sff_{i_1}\tdt \sff_{i_{k-1}}\tens \sff_2)m^{(k)}.N
=(1^{\tens(i_1+\dots+i_{k-1})}\tens m)
(\sff_{i_1}\tdt\sff_{i_{k-1}+1})m^{(k-1)},
\\
(\sff_{i_1}\tdt \sff_{i_{k-1}}\tens \sff_1)m^{(k)}.N
=(1^{\tens(i_1+\dots+i_{k-1}-1)}\tens m)
(\sff_{i_1}\tdt\sff_{i_{k-1}})m^{(k-1)},
\text{ if } i_{k-1}>1,
\\
\begin{split}
(\sff_{i_1}\tdt \sff_{i_{k-2}}\tens \sff_1\tens \sff_1)m^{(k)}.N
=(1^{\tens(i_1+\dots+i_{k-2})}\tens m)
&(\sff_{i_1}\tdt \sff_{i_{k-2}}\tens \sff_1)m^{(k-1)},
\\
&\text{ if } (i_1,\dots,i_{k-2})\ne(1,\dots,1),
\end{split}
\\
\sff_1^{\tens k}m^{(k)}.N
=(1^{\tens(k-2)}\tens m) \sff_1^{\tens(k-1)} m^{(k-1)} -m^{(k)}\sff_1.
\end{gather*}
\end{lemma}

\begin{proof}
Clearly, $\sff_1.N=0$ and for $n\ge2$
\[ \sff_n.N =\sff_n+\sum_{i+j=n}(-1)^j(\sff_i\tens\sff_j)m.h =\sff_n-\sff_n =0.
\]
For $k>1$ and $i_k>2$ we get
\begin{align*}
(\sff_{i_1}\tdt \sff_{i_k})m^{(k)}.N &=(\sff_{i_1}\tdt \sff_{i_k})m^{(k)}
+(\sff_{i_1}\tdt \sff_{i_{k-1}}\tens \sff_{i_k}.\partial)m^{(k)}.h
\\
&=(\sff_{i_1}\tdt \sff_{i_k})m^{(k)}
-(\sff_{i_1}\tdt \sff_{i_{k-1}}\tens \sff_{i_k-1}\tens \sff_1)m^{(k+1)}.h
\\
&=(\sff_{i_1}\tdt \sff_{i_k})m^{(k)} -(\sff_{i_1}\tdt \sff_{i_k})m^{(k)} =0,
\end{align*}
\begin{multline*}
(\sff_{i_1}\tdt \sff_{i_{k-1}}\tens \sff_2)m^{(k)}.N
=(\sff_{i_1}\tdt \sff_{i_{k-1}}\tens \sff_2)m^{(k)}
+(\sff_{i_1}\tdt \sff_{i_{k-1}}\tens \sff_2.\partial)m^{(k)}.h
\\
=(\sff_{i_1}\tdt \sff_{i_{k-1}}\tens \sff_2)m^{(k)}
+\bigl(\sff_{i_1}\tdt\sff_{i_{k-2}}\tens(1^{\tens(i_{k-1}-1)}\tens m) 
\sff_{i_{k-1}+1}\bigr)m^{(k-1)}
-(\sff_{i_1}\tdt \sff_{i_{k-1}}\tens \sff_2)m^{(k)}
\\
=(1^{\tens(i_1+\dots+i_{k-1})}\tens m)
(\sff_{i_1}\tdt \sff_{i_{k-2}}\tens \sff_{i_{k-1}+1})m^{(k-1)}.
\end{multline*}
For $k>1$ and $i_{k-1}>1$ we get
\begin{multline*}
(\sff_{i_1}\tdt \sff_{i_{k-1}}\tens \sff_1)m^{(k)}.N
\\
=(\sff_{i_1}\tdt \sff_{i_{k-1}}\tens \sff_1)m^{(k)}
+(\sff_{i_1}\tdt \sff_{i_{k-1}+1})m^{(k-1)}.\partial
+(\sff_{i_1}\tdt \sff_{i_{k-1}}\tens \sff_1)m^{(k)}.\partial h
\\
\begin{aligned}
=(\sff_{i_1}\tdt \sff_{i_{k-1}}\tens \sff_1)m^{(k)}
&+(\sff_{i_1}\tdt \sff_{i_{k-2}}\tens \sff_{i_{k-1}+1}.\partial)m^{(k-1)}
\\
&+(\sff_{i_1}\tdt\sff_{i_{k-2}}\tens\sff_{i_{k-1}}.\partial\tens\sff_1)
m^{(k)}.h
\end{aligned}
\\
=(\sff_{i_1}\tdt \sff_{i_{k-1}}\tens \sff_1)m^{(k)}
+\sum_{r+t=i_{k-1}-1} (-1)^t
(1^{\tens(i_1+\dots+i_{k-2}+r)}\tens m\tens1^{\tens t})
(\sff_{i_1}\tdt \sff_{i_{k-1}})m^{(k-1)}
\\
\begin{aligned}
&+\sum_{l+j=i_{k-1}+1}
(-1)^j(\sff_{i_1}\tdt \sff_{i_{k-2}}\tens \sff_l\tens \sff_j)m^{(k)}
\\
&+\sum_{r+2+t=i_{k-1}}(-1)^t
(1^{\tens(i_1+\dots+i_{k-2}+r)}\tens m\tens1^{\tens(t+1)})
(\sff_{i_1}\tdt \sff_{i_{k-1}})m^{(k-1)}
\\
&+\sum_{l+j=i_{k-1}}^{j>0}
(-1)^j(\sff_{i_1}\tdt \sff_{i_{k-2}}\tens \sff_l\tens \sff_{j+1})m^{(k)}
\end{aligned}
\\
=(1^{\tens(i_1+\dots+i_{k-1}-1)}\tens m)
(\sff_{i_1}\tdt\sff_{i_{k-1}})m^{(k-1)}.
\end{multline*}
If $k>2$ and \((i_1,\dots,i_{k-2})\ne(1,\dots,1)\), then
\begin{multline*}
(\sff_{i_1}\tdt \sff_{i_{k-2}}\tens \sff_1\tens \sff_1)m^{(k)}.N
\\
=(\sff_{i_1}\tdt \sff_{i_{k-2}}\tens \sff_1\tens \sff_1)m^{(k)}
+(\sff_{i_1}\tdt \sff_{i_{k-2}}\tens \sff_2)m^{(k-1)}.\partial
+(\sff_{i_1}\tdt \sff_{i_{k-2}}\tens \sff_1\tens \sff_1)m^{(k)}.\partial h
\\
=(\sff_{i_1}\tdt \sff_{i_{k-2}}\tens \sff_1\tens \sff_1)m^{(k)}
+(\sff_{i_1}\tdt \sff_{i_{k-2}}\tens \sff_2.\partial)m^{(k-1)}
\\
=(\sff_{i_1}\tdt \sff_{i_{k-2}}\tens \sff_1\tens \sff_1)m^{(k)}
+(\sff_{i_1}\tdt \sff_{i_{k-2}}\tens m\sff_1)m^{(k-1)}
-(\sff_{i_1}\tdt \sff_{i_{k-2}}\tens \sff_1\tens \sff_1)m^{(k)}
\\
=(1^{\tens(i_1+\dots+i_{k-2})}\tens m)
(\sff_{i_1}\tdt \sff_{i_{k-2}}\tens \sff_1)m^{(k-1)}.
\end{multline*}
For $k\ge2$ we have
\begin{multline*}
\sff_1^{\tens k}m^{(k)}.N =\sff_1^{\tens k}m^{(k)} -m^{(k)}\sff_1
+(\sff_1^{\tens(k-2)}\tens \sff_2)m^{(k-1)}.\partial
\\
=\sff_1^{\tens k}m^{(k)}-m^{(k)}\sff_1 +(\sff_1^{\tens(k-2)}\tens m\sff_1)m^{(k-1)}
-\sff_1^{\tens k}m^{(k)}
=(1^{\tens(k-2)}\tens m) \sff_1^{\tens(k-1)} m^{(k-1)} -m^{(k)}\sff_1.
\end{multline*}
These cases cover all generators of \(\overline{F}_1\).
\end{proof}

The proposition is proven.
\end{proof}

\begin{remark}\label{rem-A8-123-abc}
Recall that the operadic part $A_\infty$ of \((A_\infty,F_1)\) has a 
$\kk$\n-basis indexed by isomorphism classes of ordered rooted trees without unary vertices.
The bimodule part $F_1$ has a $\kk$\n-basis indexed by isomorphism classes of trees $t$, \(\Inp(t)=L(t)\), with a distinguished floor \(f(t)\subset\IV(t)\)  that satisfy the following requirements:
\begin{enumerate}
\renewcommand{\labelenumi}{(\arabic{enumi})}
\item the set \(\IV(t)-f(t)\) contains no unary vertices;

\item the set \(f(t)\) contains no nullary vertices.
\end{enumerate}
In this correspondence the following vertices are assigned to factors of an element of \(F_1=A_\infty\odot\kk\{\sff_n\mid n\ge1\}\odot A_\infty\):
\begin{enumerate}
\renewcommand{\labelenumi}{(\alph{enumi})}
\item an operation $m_n$ from the left factor $A_\infty$ is represented by an $n$\n-corolla \(u\in\IV(t)-f(t)\) such that there is \(v\in f(t)\) for which $u\preccurlyeq v$ ($u$ is above the distinguished floor);

\item an operation $\sff_n$ from the middle factor is represented by an 
$n$\n-corolla \(v\in f(t)\) ($v$ is on the distinguished floor);

\item an operation $m_n$ from the right factor $A_\infty$ is represented by an $n$\n-corolla \(w\in\IV(t)-f(t)\) such that there is \(v\in f(t)\) for which $v\preccurlyeq w$ ($w$ is below the distinguished floor).
\end{enumerate}
\end{remark}

\begin{example}
Algebra maps over the $\dg$\n-operad bimodule \((A_\infty,F_1)\) are precisely morphisms of \ainf-algebras in the usual sense: a pair $A$, $B$ of 
\ainf-algebras and $\kk$\n-linear maps \(\sff_n:A^{\tens n}\to B\), $n\ge1$, of degree \(\deg \sff_n=1-n\), such that
\begin{gather*}
\sum^{l>0}_{i_1+\dots+i_l=k} (-1)^\sigma 
(\sff_{i_1}\tens \sff_{i_2}\tdt \sff_{i_l}) m_l
=\sum_{r+n+t=k} (-1)^{t+rn}(1^{\tens r}\tens m_n\tens1^{\tens t}) \sff_{r+1+t},
\end{gather*}
where $\sigma$ is given by \eqref{eq-sigma-sign}.
\end{example}

\subsection{Unital \texorpdfstring{$A_\infty$}{A8}-morphisms.}
Strictly unital \ainf-morphisms are governed by the operad bimodule $(A_\infty^\su,F_1^\su)$ generated over \((A_\infty,F_1)\) by a cycle
$1^\su\in A_\infty^\su(0)^0$ subject to the following relations:
\begin{equation}
1^\su\rho_\emptyset =1^\su \sff_1, \quad
(1^{\tens a}\tens1^\su\tens1^{\tens b})\sff_{a+1+b} =0 \text{ \ if \ } a+b>0.
\label{eq-1surho-1suf1}
\end{equation}
The rows of the following diagram in \(\dg^{\NN\sqcup\NN}\)
\begin{diagram}[LaTeXeqno]
0 &\rTTo &(A_\infty,F_1) &\rTTo &(A_\infty^\su,F_1^\su) &\rTTo
&(\kk1^\su,\kk1^\su\rho_\emptyset) &\rTTo &0
\\
&& \dFib<{htis} &&\dFib<{htis}>{p'} &&\dEq
\\
0 &\rTTo &(\As,\As) &\rTTo &(\Ass,\Ass) &\rTTo 
&(\kk1^\su,\kk1^\su\rho_\emptyset) &\rTTo &0
\label{dia-(AF)-(AsuFsu)-(k1k1)}
\end{diagram}
are exact sequences, split in the obvious way. 
Therefore, the middle vertical arrow $p'$ is a homotopy isomorphism.

Just as before, we add to the operadic part of $(A_\infty^\su,F_1^\su)$ two nullary operations $i$, $j$, $\deg i=0$, $\deg j=-1$, with \(i\partial=0\), \(j\partial=1^\su-i\). 
The projection $p'$ decomposes into a standard trivial cofibration and an epimorphism $p''$
\[ p' =\bigl( (A_\infty^\su,F_1^\su) \rCof~\Sim^{htis}
(A_\infty^\su,F_1^\su)\langle i,j\rangle \rTTo^{p''} (\Ass,\Ass) \bigr),
\]
where \(p''(1^\su)=1^\su\), \(p''(i)=1^\su\), \(p''(m_2)=m_2\), \(p''(\sff_1)=1^F\), and other generators go to 0. 
As a corollary \(p''(1^\su\rho_\emptyset)=1^\su\rho_\emptyset\). 
Hence, the projection $p''$ is a homotopy isomorphism as well.

A cofibrant replacement \((A_\infty^{hu},F_1^{hu})\to(\Ass,\Ass)\) is constructed as a $\dg$-subbimodule of 
\((A_\infty^\su,F_1^\su)\langle i,j\rangle\) generated in operadic part by $i$ and $n$\n-ary operations of degree $4-n-2k$
\[ m_{n_1;n_2;\dots;n_k} =(1^{\tens n_1}\tens j\tens1^{\tens n_2}\tens
j\tdt1^{\tens n_{k-1}}\tens j\tens1^{\tens n_k})m_{n+k-1},
\]
where \(n=\sum_{q=1}^kn_q\), $k\ge1$, $n_q\ge0$, \(n+k\ge3\) and in bimodule part by the nullary element \(v=j\sff_1-j\rho_\emptyset\), \(\deg v=-1\), and $n$\n-ary elements of degree $3-n-2k$
\begin{equation}
\sff_{n_1;n_2;\dots;n_k} =(1^{\tens n_1}\tens j\tens1^{\tens n_2}\tens
j\tdt1^{\tens n_{k-1}}\tens j\tens1^{\tens n_k})\sff_{n+k-1},
\label{eq-fn1n2nk}
\end{equation}
with \(n+k\ge2\) except for $\sff_{0;0}$ which is a summand of $v$. 
Note that \(v\partial=i\rho_\emptyset-i\sff_1\).

Let us prove that the graded bimodule $(A_\infty^{hu},F_1^{hu})$ is free over $(\kk,0)$. 
The graded bimodule \((A_\infty,F_1)\langle j\rangle\) can be presented as
\begin{multline}
(A_\infty, A_\infty\odot\kk\{\sff_n \mid n\ge1\}\odot A_\infty)\langle j\rangle
\\
\simeq
\bigl( A_\infty\langle j\rangle, 
\kk\langle m_{n_1;\dots;n_k} \mid {k+\ttt\sum_{q=1}^k} n_q\ge3\rangle \odot \kk\{\sff_{n_1;\dots;n_k} \mid {k+\ttt\sum_{q=1}^k} n_q\ge2\} 
\odot A_\infty\langle j\rangle \bigr)
\label{eq-(AAokfoA)j}
\end{multline}
taking into account \eqref{eq-fn1n2nk}.
The free graded operad bimodule generated by \(m_{n_1;\dots;n_k}\) and \(\sff_{n_1;\dots;n_k}\) has the form
\begin{multline*}
K= T\bigl(\kk\{m_{n_1;\dots;n_k} \mid {k+\ttt\sum_{q=1}^k} n_q\ge3\},
\kk\{\sff_{n_1;\dots;n_k} \mid {k+\ttt\sum_{q=1}^k} n_q\ge2\}\bigr)
\\
\hskip\multlinegap
=\bigl(\kk\langle m_{n_1;\dots;n_k} \mid {k+\ttt\sum_{q=1}^k} n_q\ge3\rangle,
\hfill
\\
\kk\langle m_{n_1;\dots;n_k} \mid {k+\ttt\sum_{q=1}^k} n_q\ge3\rangle\odot
\kk\{\sff_{n_1;\dots;n_k} \mid {k+\ttt\sum_{q=1}^k} n_q\ge2\}\odot
\kk\langle m_{n_1;\dots;n_k} \mid {k+\ttt\sum_{q=1}^k} n_q\ge3\rangle\bigr).
\end{multline*}
It is a direct summand of \eqref{eq-(AAokfoA)j}, so we have a split exact sequence in \(\gr^{\NN\sqcup\NN}\)
\[ 0 \rTTo K \pile{\rTTo^\alpha \\ \lTTo_\pi} (A_\infty,F_1)\langle j\rangle
\pile{\rTTo^\varkappa \\ \lTTo_\omega} (\kk j,\kk j\rho_\emptyset) \rTTo 0,
\]
where $\omega$ takes $j\rho_\emptyset$ to the nullary generator
 \(j\rho_\emptyset=1\tens1\tens j
 \in\kk\tens\kk\tens A_\infty\langle j\rangle(0)\).
Consider also the graded bimodule
\[ L= 
T\bigl(\kk\{m_{n_1;\dots;n_k} \mid {k+\ttt\sum_{q=1}^k} n_q\ge3\},
\kk\{v,\sff_{n_1;\dots;n_k} \mid {k+\ttt\sum_{q=1}^k} n_q\ge3 
\text{ \ or \ } k=1=n_1 \}\bigr).
\]
Notice that the map $L\to K$, \(v\mapsto\sff_{0;0}\), identifies the bimodules $L$ and $K$.
The bimodule part of $K$ (or $L$) has a basis \((b(t,\Inp(t),f(t)))\) indexed by trees $t$ with a distinguished floor \(f(t)\subset\IV(t)\) such that conditions (1), (2) of \remref{rem-A8-123-abc} hold.
Elements \(b(t,\Inp(t),f(t))\) are obtained similarly to (a)--(c) of \remref{rem-A8-123-abc} using $m_{n_1;\dots;n_k}$ and $\sff_{n_1;\dots;n_k}$ in place of $m_n$, $\sff_n$.
For simplicity we represent $v$ in $L$ by 
 \ $
\hstretch140
\vstretch50
\begin{tanglec}
\noder{j}
\\
\s
\\
\sd
\end{tanglec}
 $ \
just as $\sff_{0;0}$ in $K$, not as 
 \ $
\hstretch140
\vstretch70
\begin{tanglec}
\noder{v}
\\
\sd
\end{tanglec}
 $ \ ,
which would require modification of conditions (1), (2). 
Besides both presentations are interchangeable.

Consider the graded bimodule morphism
\[ \beta: L \to (A_\infty,F_1)\langle j\rangle, \qquad 
v \mapsto \sff_{0;0} -j\rho_\emptyset,
\]
which maps other generators identically.
The morphism $\beta$ extends to basic elements so that each factor $j\rho_\emptyset$ arising from a vertex of type $v$ gives its $j$ to subsequent $m_{;;;}$ adding another semicolon to its indexing sequence.
This follows by associativity of $\rho$.
The basic element $v$ is mapped by \(\beta\varkappa\) to $-j\rho_\emptyset$.
For any other basic element \(b(t)\in L\) we have \(b(t).\beta\varkappa=0\).

The map \(\beta-\beta\varkappa\omega\in\gr^{\NN\sqcup\NN}\) factors through $\alpha$ as the following diagram shows:
\begin{diagram}
0 &\rTTo &K &\pile{\rTTo^\alpha \\ \lTTo_\pi} &(A_\infty,F_1)\langle j\rangle
&\pile{\rTTo^\varkappa \\ \lTTo_\omega} &(\kk j,\kk j\rho_\emptyset) &\rTTo &0
\\
&&\uTTo<{\exists!\gamma} &\ruTTo>{\beta-\beta\varkappa\omega} &
\\
&&L
\end{diagram}

The unique map
 \(\gamma=(\beta-\beta\varkappa\omega)\pi=\beta\pi:
 L\to K\in\gr^{\NN\sqcup\NN}\), 
such that \(\beta-\beta\varkappa\omega=\gamma\alpha\), has a triangular matrix.
In fact, \(v.\gamma=\sff_{0;0}\), and $L$ and $K$ have an $\NN$\n-grading, \(L^q=\oplus\kk b(t)\), \(K^q=\oplus\kk b(t)\), where the summation is over trees $t$ with $q$ vertices of type $v$ (resp. \(\sff_{0;0}\)).
The map $\gamma$ takes the filtration \(L_q=L^0\oplus\dots\oplus L^q\) to the filtration \(K_q=K^0\oplus\dots\oplus K^q\). 
The diagonal entries \(\gamma^{qq}:L^q\to K^q\) are identity maps.
Thus, the matrix of \(\gamma\) equals \(1-N\), where $N$ is locally nilpotent, and $\gamma$ is invertible.
We obtained a split exact sequence
\[ 0 \rTTo L \rTTo^{\beta-\beta\varkappa\omega} (A_\infty,F_1)\langle j\rangle
\pile{\rTTo^\varkappa \\ \lTTo_\omega} (\kk j,\kk j\rho_\emptyset) \rTTo 0.
\]
The map $\beta$ (and \(\beta-\beta\varkappa\omega\)) takes the direct complement \(L\ominus\kk v\) spanned by all basic elements except $v$ to the direct complement
\((A_\infty,F_1)\langle j\rangle\ominus\kk\{\sff_{0;0},j\rho_\emptyset\}\).
Therefore, this restriction of $\beta$ is an isomorphism and we can extend it to a split exact sequence
\[ 0 \rTTo L \rTTo^\beta (A_\infty,F_1)\langle j\rangle
\pile{\rTTo^\theta \\ \lTTo_\omega} (\kk j,\kk j\rho_\emptyset) \rTTo 0,
\]
where \(\sff_{0;0}.\theta=j\rho_\emptyset.\theta=j\rho_\emptyset\) and on the complement $\theta$ vanishes.

Adding freely $i$ we deduce the split exact sequence in \(\gr^{\NN\sqcup\NN}\)
\begin{equation}
0 \to L\langle i\rangle \to (A_\infty,F_1)\langle j,i\rangle
\to (\kk j,\kk j\rho_\emptyset) \to 0.
\label{eq-Li-(AF)ji-(kjkjr)}
\end{equation}
The image of the embedding is precisely $(A_\infty^{hu},F_1^{hu})$.

On the other hand, from the top row of diagram~\eqref{dia-(AF)-(AsuFsu)-(k1k1)} we deduce a splittable exact sequence in \(\gr^{\NN\sqcup\NN}\)
\[ 0 \to (A_\infty,F_1)\langle i\rangle \to 
(A_\infty^\su,F_1^\su)\langle i\rangle \to
(\kk1^\su,\kk1^\su\rho_\emptyset) \to 0.
\]
We may choose the splitting of this exact sequence as indicated below:
\[ 0 \to (A_\infty,F_1)\langle i\rangle \to 
(A_\infty^\su,F_1^\su)\langle i\rangle \to
(\kk\{1^\su-i\},\kk\{1^\su\rho_\emptyset-i\rho_\emptyset\}) \to 0.
\]
Adding freely $j$ we get the split exact sequence
\begin{equation}
0 \to (A_\infty,F_1)\langle i,j\rangle \to 
(A_\infty^\su,F_1^\su)\langle i,j\rangle \to
(\kk\{1^\su-i\},\kk\{(1^\su-i)\rho_\emptyset\}) \to 0.
\label{eq-(AF)ij-(AsuFsu)ij-(k1ik1ir)}
\end{equation}
Combining \eqref{eq-Li-(AF)ji-(kjkjr)} with \eqref{eq-(AF)ij-(AsuFsu)ij-(k1ik1ir)} we get a split exact sequence 
\begin{equation}
0 \to (A_\infty^{hu},F_1^{hu}) \xrightarrow{i'}
(A_\infty^\su,F_1^\su)\langle i,j\rangle \to
(\kk\{1^\su-i,j\},\kk\{(1^\su-i)\rho_\emptyset,j\rho_\emptyset\}) \to 0.
\label{eq-0-(Ahu8F1hu)-ij-0}
\end{equation}

The differential in $(A_\infty^{hu},F_1^{hu})$ is computed through that of \((A_\infty^\su,F_1^\su)\langle i,j\rangle\). 
Actually, \eqref{eq-0-(Ahu8F1hu)-ij-0} is a split exact sequence in \(\dg^{\NN\sqcup\NN}\), where the third term obtains the differential 
\(j.\partial=1^\su-i\), 
\(j\rho_\emptyset.\partial=1^\su\rho_\emptyset-i\rho_\emptyset\). 
The third term is contractible, which shows that the inclusion $i'$ is a homotopy isomorphism in \(\dg^{\NN\sqcup\NN}\). 
Hence, the epimorphism \(p=i'\cdot p'':(A_\infty^{hu},F_1^{hu})\to(\Ass,\Ass)\) is a homotopy isomorphism as well.

In order to prove that $(\1,0)\to(A_\infty^{hu},F_1^{hu})$ is a standard cofibration we present it as a colimit of sequence of elementary cofibrations \((\1,0)\to\cd_0=T(\kk\{i,m_2\},\kk \sff_1)\to\cd_1\to\cd_2\to\dots\), for $r>0$
\[ \cd_r =T\bigl(
\kk\{i,m_{n_1;\dots;n_k} \mid \deg m_{n_1;\dots;n_k} \ge-r\},
\kk\{v,\sff_{n_1;\dots;n_k} \mid \deg \sff_{n_1;\dots;n_k} \ge-r\} \bigr).
\]

Algebra maps over $(A_\infty^{hu},F_1^{hu})$ are identified with homotopy unital \ainf-morphisms, which we define in the spirit of Fukaya's approach:

\begin{definition}
A \emph{homotopy unital structure} of an \ainf-morphism $f:A\to B$ is an
\ainf-morphism 
$f^+:A^+=A\oplus\kk1^\su_A\oplus\kk j^A\to B\oplus\kk1^\su_B\oplus\kk j^B=B^+$ such that:
\begin{enumerate}
\renewcommand{\labelenumi}{(\arabic{enumi})}
\item $f^+$ is a strictly unital:
\[  1^\su_A\sff_1=1^\su_B, \quad
(1^{\tens a}\tens1^\su_A\tens1^{\tens b})\sff_{a+1+b} =0 \text{ \ if \ } a+b>0;
\]

\item the element \(v^B=j^A\sff^+_1-j^B\) is contained in $B$;

\item the restriction of $f^+$ to $A$ gives $f$;

\item \((A\oplus\kk j)^{\tens n}\sff^+_n\subset B\), for each $n>1$.
\end{enumerate}
\end{definition}

Note that Fukaya, Oh, Ohta and Ono define homotopy-unital $A_\infty$ homomorphisms using only above conditions (1) and (3) in their book \cite[Definition~3.3.13]{FukayaOhOhtaOno:Anomaly}.

Homotopy unital structure of an \ainf-morphism $f$ means a \emph{choice} of such $f^+$. 
There is another notion of unitality which is a \emph{property} of an 
\ainf-morphism:

\begin{definition}[{\cite[Definition~8.1]{Lyu-AinfCat}}]
An \ainf-morphism $f:A\to B$ between unital \ainf-algebras is \emph{unital} if the cycles \(i^A\sff_1\) and \(i^B\) differ by a boundary.
\end{definition}

For a homotopy unital \ainf-morphism $f:A\to B$ the equation holds \(v^Bm_1=v\partial=i\rho_\emptyset-i\sff_1=i^B-i^A\sff_1\). 
Thus a homotopy unital \ainf-morphism is unital.

\begin{conjecture}
Unitality of an \ainf-morphism is equivalent to homotopy unitality: any unital \ainf-morphism admits a homotopy unital structure.
\end{conjecture}

\subsection{Composition of unital \texorpdfstring{$A_\infty$}{A8}-morphisms.}
The category of collections of complexes is cocomplete and the tensor product $\odot$ preserves colimits, since $\odot$ is obtained from direct sums of the right exact functor $\tens_\kk$. 
Therefore, operad bimodules inherit from $\odot$ a bicategory structure. 
The composition functor
\[ \co_0\imodi\co_1\times\co_1\imodi\co_2\times\dots\times \co_{n-1}\imodi\co_n
\to \co_0\imodi\co_n
\]
takes a sequence of bimodules \((\cp_1,\dots,\cp_n)\), $n\ge1$, to
\[ \cp_1\odot_{\co_1}\cp_2\odot_{\co_2}\dots\odot_{\co_{n-1}}\cp_n
=\colim\bigl( \cp_1\odot\co_1\odot\cp_2\odot\dots\odot\co_{n-1}\odot\cp_n
\rTTo^{2^{n-1}\textup{ arrows}} \cp_1\odot\cp_2\odot\dots\odot\cp_n \bigr).
\]
For $n>1$ each of $2^{n-1}$ arrows is determined by choosing whether $\co_i$ acts on the right on $\cp_i$ or on the left on $\cp_{i+1}$ for all $i\le n-1$. 
For $n=1$ the above functor is the identity functor. 
For $n=0$ the unit 1\n-morphisms \(\1\to\co\imodi\co\) send the only object of the terminal category to the regular $\co$\n-bimodule.

In order to have an associative composition of \((\co,\cf)\)-morphisms, we postulate that $\cf$ is equipped with a coassociative counital coalgebra structure \((\cf,\Delta:\cf\to\cf\odot_\co\cf,\eps:\cf\to\co)\) in the monoidal category of $\co$\n-bimodules. 
Compatibility with the right action \(\rho_{n_1,\dots,n_k}\) takes for $k=0$ the form of the equation
\begin{diagram}[LaTeXeqno]
\co(0) &\rTTo^{\rho_\emptyset} &\cf(0)
\\
\dTTo<1 &= &\dTTo>\Delta
\\
\co(0) &\rTTo^{\rho_\emptyset} &(\cf\odot_\co\cf)(0)
\label{dia-rho-emptyset}
\end{diagram}

Suppose that $A$, $B$, $C$ are $\co$\n-algebras and $g:\cf\to\HOM(A,B)$, 
$h:\cf\to\HOM(B,C)$ are \((\co,\cf)\)-morphisms. 
Then their composition is defined as the convolution
\[ g\cdot h =\bigl[ \cf \rTTo^\Delta \cf\odot_\co\cf \rTTo^{g\odot h}
\HOM(A,B)\odot_{\END B}\HOM(B,C) \to \HOM(A,C) \bigr].
\]
The convolution composition is associative. 
The unit \((\co,\cf)\)-morphism of an $\co$\n-algebra $B$ is
\[ 1_B =\bigl( \cf \rTTo^\eps \co \to \END B \bigr).
\]
In this way $\co$\n-algebras and their \((\co,\cf)\)-morphisms form a category.

For \((\As,\As)\) and \((\Ass,\Ass)\) the comultiplication and the counit are the identity maps. 
For \ainf-morphisms the comultiplication 
\(\Delta:F_1\to F_1\odot_{A_\infty}F_1\) is
\[ \sff_n\Delta
=\sum_{i_1+\dots+i_k=n}(\sff_{i_1}\tens \sff_{i_2}\tdt \sff_{i_k})\tens \sff_k,
\]
which results in the composition
\begin{equation}
(g\cdot h)_n=\sum_{i_1+\dots+i_k=n}(g_{i_1}\tens g_{i_2}\tdt g_{i_k})\tens h_k.
\label{eq-ghn-(gggg)h}
\end{equation}
The counit \(\eps:F_1\to A_\infty\) sends \(\sff_1\in F_1(1)\) to the operad unit \(1\in A_\infty(1)\) and other generators $\sff_n$, $n>1$ to 0. 
Hence the identity \ainf-morphism $\id^B:B\to B$ has $\id^B_1=1$ and $\id^B_n=0$ for $n>1$.

Exactly the same formulae for differential, comultiplication and counit define a coalgebra structure for the free \(A_\infty^\su\)-bimodule generated by $\sff_n$, $n\ge1$. 
One can verify that its quotient bimodule \(F_1^\su\) inherits this coalgebra structure, namely, that the subbimodule of relations generated by \eqref{eq-1surho-1suf1} is a coideal. 
Informally it amounts to check that composition of strictly unital 
\ainf-morphisms is again strictly unital.

When the elements $i$, $j$ are freely added to the operadic part of the bimodule \((A_\infty^\su,F_1^\su)\) the coalgebra structure is preserved. 
It takes the form
 \((F_1^\su\langle i,j\rangle,\Delta:F_1^\su\langle i,j\rangle\to
 F_1^\su\langle i,j\rangle\odot_{A_\infty^\su\langle i,j\rangle}
 F_1^\su\langle i,j\rangle,
 \eps:F_1^\su\langle i,j\rangle\to A_\infty^\su\langle i,j\rangle)\).
The subbimodule
\((A_\infty^{hu},F_1^{hu})\subset(A_\infty^\su,F_1^\su)\langle i,j\rangle\) is actually a subcoalgebra, namely, the unique top arrows in commutative squares
\[
\begin{diagram}[inline]
F_1^{hu} &\rTTo^\Delta &F_1^{hu}\odot_{A_\infty^{hu}}F_1^{hu}
\\
\dMono &&\dMono>\alpha
\\
F_1^\su\langle i,j\rangle &\rTTo^\Delta
&F_1^\su\langle i,j\rangle\odot_{A_\infty^\su\langle i,j\rangle}
F_1^\su\langle i,j\rangle
\end{diagram}
\hspace{4em},\hspace{2em}
\begin{diagram}[inline]
F_1^{hu} &\rTTo^\eps &A_\infty^{hu}
\\
\dMono &&\dMono
\\
F_1^\su\langle i,j\rangle &\rTTo^\eps
&A_\infty^\su\langle i,j\rangle
\end{diagram}
\]
turn $F_1^{hu}$ into a coalgebra. Proof of this fact uses 
equation~\eqref{dia-rho-emptyset}:
\[ v\Delta =(j\sff_1-j\rho_\emptyset)\Delta 
=j\sff_1\tens \sff_1 -j\rho_\emptyset
=(j\sff_1-j\rho_\emptyset)\tens\sff_1 +j\sff_1-j\rho_\emptyset =v\tens\sff_1 +v
\]
as well as definition of the tensor product of bimodules. 
On other generators we transform
\[ \sff_{n_1;n_2;\dots;n_k}\Delta =(1^{\tens n_1}\tens j\tens1^{\tens n_2}\tens
j\tdt1^{\tens n_{k-1}}\tens j\tens1^{\tens n_k})
\sum_{i_1+\dots+i_l=n+k-1} (\sff_{i_1}\tens\sff_{i_2}\tdt\sff_{i_l})\tens\sff_l
\]
as follows. Each $\sff_{i_q}$ is replaced with a generator 
\(\sff_{a_1;a_2;\dots;a_p}\) accordingly with the set of $j$'s appearing among the arguments of $\sff_{i_q}$. 
The only exception is the case of $j\sff_1$ which is replaced with $v+j\rho_\emptyset$. 
In obtained summands all instances of $j\rho_\emptyset$ are moved to the right as arguments $j$ of $\sff_l$ due to defining the tensor product as a colimit. 
Finally, $\sff_l$ is replaced with a generator \(\sff_{c_1;c_2;\dots;c_t}\).

In order to show that the arrow $\alpha$ is a monomorphism we present it as the composition
\[ \alpha =\bigl( F_1^{hu}\odot_{A_\infty^{hu}}F_1^{hu} \rTTo^{1\odot i'}
F_1^{hu}\odot_{A_\infty^{hu}}(F_1^\su\langle i,j\rangle) \rTTo^\beta
(F_1^\su\langle i,j\rangle)\odot_{A_\infty^\su\langle i,j\rangle}
(F_1^\su\langle i,j\rangle) \bigr).
\]
Such $i'$ is a split monomorphism by \eqref{eq-0-(Ahu8F1hu)-ij-0}, and so is $1\odot i'$. 
By the way, 
\(\Coker1\odot i'=\Coker i'=\kk\{(1^\su-i)\rho_\emptyset,j\rho_\emptyset\}\). 
The morphism $\beta$ is a particular case of a morphism
\[ \beta =\bigl( F_1^{hu}\odot_{A_\infty^{hu}}\cp \rMono^{i'\odot1}
(F_1^\su\langle i,j\rangle)\odot_{A_\infty^{hu}}\cp \rEpi
(F_1^\su\langle i,j\rangle)\odot_{A_\infty^\su\langle i,j\rangle}\cp \bigr),
\]
which is invertible for an arbitrary left 
\(A_\infty^\su\langle i,j\rangle\)-module $\cp$. 
In fact, split exact sequence \eqref{eq-0-Ahu-Asuij-k1ij-0} shows that \(\odot_{A_\infty^\su\langle i,j\rangle}\) is obtained from \(\odot_{A_\infty^{hu}}\) by extra identifications: \((1^\su-i)\rho_\emptyset\) and \(j\rho_\emptyset\) can be moved from the first factor to $\cp$, where they operate through $\lambda$. 
It remains to notice that 
\[ F_1^\su\langle i,j\rangle
=F_1^{hu}\oplus\kk\{(1^\su-i)\rho_\emptyset,j\rho_\emptyset\}
\]
due to split exact sequence \eqref{eq-0-(Ahu8F1hu)-ij-0}.

This comultiplication leads to composition~\eqref{eq-ghn-(gggg)h} of
\ainf-morphisms $g^+:A^+\to B^+$ and $h^+:B^+\to C^+$. 
For instance, \(v\Delta=v+v\tens \sff_1\) corresponds to the identity
\[ j^Ag^+_1h^+_1 =(j^B +v^B)h^+_1 =j^C +v^C +v^Bh_1.
\]

\subsection{Concluding remarks.}
There are at least three different approaches to \ainf-categories and their particular species -- \ainf-algebras. 
In one approach \ainf-categories and algebras are represented as differential graded tensor coalgebras. 
Another approach exploits the fact that \ainf-categories form a closed multicategory \cite{BesLyuMan-book}. 
The third approach presented in the current article is based on resolutions (cofibrant replacements) of $\dg$-operads and bimodules over them. 
It would be very instructive to combine all three approaches into a single one. 
This could lead to a better understanding of the subject staying behind the \ainf-notions.

    \ifx\chooseClass2

	\else

	\fi

    \ifx\chooseClass2
\bibliographystyle{elsarticle-num}
	\else
\bibliographystyle{amsalpha}
\tableofcontents
	\fi


\begin{thebibliography}{FOOO09}

\bibitem[BLM08]{BesLyuMan-book}
Yu.~Bespalov, V.~V. Lyubashenko, and O.~Manzyuk, \emph{Pretriangulated
  ${A}_\infty$-categories}, Proceedings of the Inst. of Mathematics NASU,
  vol.~76, Inst. of Mathematics, Nat. Acad. Sci. Ukraine, Kyiv, 2008, 599~p.
  \url{http://www.math.ksu.edu/~lub/papers.html}.

\bibitem[FOOO09]{FukayaOhOhtaOno:Anomaly}
K.~Fukaya, Y.-G. Oh, H.~Ohta, and K.~Ono, \emph{Lagrangian intersection
  {F}loer theory: anomaly and obstruction}, AMS/IP Studies in
  Adv. Math. Series, Amer. Math. Soc., 2009, 805~p.

\bibitem[Fuk02]{Fukaya:FloerMirror-II}
K.~Fukaya, \emph{Floer homology and mirror symmetry. {II}}, Minimal
  surfaces, geometric analysis and symplectic geometry (Baltimore, MD, 1999), 
  Adv. Stud. Pure Math., vol.~34, Math. Soc. Japan, Tokyo, 2002, pp.~31--127.

\bibitem[Hin97]{Hinich:q-alg/9702015}
V.~Hinich, \emph{Homological algebra of homotopy algebras}, Comm. Algebra
  \textbf{25} (1997), no.~10, 3291--3323,
  \href{http://arXiv.org/abs/q-alg/9702015}{{\tt
  arXiv:\linebreak[1]q-alg/9702015}}.

\bibitem[KS09]{math.RA/0606241}
M.~Kontsevich and Y.~S. Soibelman, \emph{Notes on ${A}_\infty$-algebras,
  ${A}_\infty$-categories and non-commutative geometry. {I}}, Homological
  Mirror Symmetry: New Developments and Perspectives (A.~Kapustin et~al., 
  eds.), Lecture Notes in Physics, vol. 757, Springer, Berlin, Heidelberg,
  2009, pp.~153--219,
  \href{http://arXiv.org/abs/math.RA/0606241}{{\tt
  arXiv:\linebreak[1]math.RA/0606241}}.

\bibitem[LM06]{LyuMan-unitalAinf}
V.~V. Lyubashenko and O.~Manzyuk, \emph{Unital ${A}_\infty$-categories},
  Problems of topology and related questions (V.~V. Sharko, ed.), vol.~3, 
  Proc. of Inst. of Mathematics NASU, no.~3, Inst. of Mathematics, Nat. Acad.
  Sci. Ukraine, Kyiv, 2006, pp.~235--268,
  \href{http://arXiv.org/abs/0802.2885}{{\tt arXiv:\linebreak[1]0802.2885}}.

\bibitem[Lyu03]{Lyu-AinfCat}
V.~V. Lyubashenko, \emph{Category of ${A}_\infty$-categories}, Homology,
  Homotopy Appl. \textbf{5} (2003), no.~1, 1--48,
  \href{http://arXiv.org/abs/math.CT/0210047}{{\tt
  arXiv:\linebreak[1]math.CT/0210047}}
  \url{http://intlpress.com/HHA/v5/n1/a1/}.

\bibitem[Mar96]{Markl:ModOp}
M.~Markl, \emph{Models for operads}, Commun. in Algebra \textbf{24} (1996),
  no.~4, 1471--1500, \href{http://arXiv.org/abs/hep-th/9411208}{{\tt
  arXiv:\linebreak[1]hep-th/9411208}}.

\bibitem[Sei08]{SeidelP-book-Fukaya}
P.~Seidel, \emph{Fukaya categories and {P}icard--{L}efschetz theory}, Z\"urich
  Lectures in Adv. Math., European Math. Soc. (EMS), Z\"urich, 2008, 
  viii+326~p.

\bibitem[Sta63]{Stasheff:HomAssoc}
J.~D. Stasheff, \emph{Homotopy associativity of {H}-spaces {I} $\&$ {II}},
  Trans. Amer. Math. Soc. \textbf{108} (1963), 275--292, 293--312.

\end{thebibliography}

\end{document}